\documentclass[11pt,reqno,bold]{paper}
\usepackage{geometry}
\usepackage{amsmath}
\usepackage{amssymb}
\usepackage{amsfonts}
\usepackage{mathrsfs}
\usepackage{dsfont}
\usepackage{enumerate}
\usepackage{comment}
\usepackage{amsthm}
\usepackage{color}
\definecolor{darkblue}{rgb}{0.0,0.0,0.4}
\usepackage[pdfauthor={Al Balushi, Tsogtgerel},pdftitle={Adaptive finite element method for the biharmonic problem},pdfsubject={adaptive methods, convergence rate, biharmonic problem},colorlinks,citecolor=darkblue,linkcolor=darkblue,urlcolor=darkblue]{hyperref}

\newtheorem{theorem}{Theorem}[section]

\newtheorem{lemma}[theorem]{Lemma}
\newtheorem{corollary}[theorem]{Corollary}
\theoremstyle{definition}
\theoremstyle{remark}
\newtheorem{remark}[theorem]{Remark}

\usepackage{algorithmicx}
\usepackage[ruled]{algorithm}
\usepackage{algpseudocode}

\excludecomment{comment:quasi}

\newcommand{\ms}[1]{{\mathbf{#1}}}

\newcommand{\ds}[1]{{\mathds{#1}}}

\renewcommand{\rm}[1]{{\mathrm{#1}}}
\newcommand{\bb}[1]{{\mathbb{#1}}}
\renewcommand{\cal}[1]{{\mathcal{#1}}}
\newcommand{\scr}[1]{{\mathscr{#1}}}
\newcommand{\n}{\ms{n}}
\newcommand{\diff}{\partial}
\newcommand{\grad}{\nabla}
\newcommand{\Lap}{\Delta}

\renewcommand{\dfrac}[2]{\frac{\diff #1}{\diff #2}}
\newcommand{\semi}[2]{{|#1|}_{#2}}
\newcommand{\norm}[2]{\|#1\|_{#2}}
\newcommand{\inner}[1]{\left\langle#1\right\rangle}
\newcommand{\enorm}[2]{|\!|\!|#1|\!|\!|_{#2}}
\newcommand{\jump}[2]{\bb{J}_{#2}\!\left(#1\right)}

\newcommand{\supp}{\mathrm{supp}\,}
\newcommand{\diam}{\rm{diam}\,}
\newcommand{\lapprox}{\preceq}

\newcommand{\eps}{\varepsilon}
\renewcommand{\u}{u}
\renewcommand{\v}{v}

\newcommand{\e}{e}
\newcommand{\U}{U}
\newcommand{\V}{V}
\newcommand{\W}{W}
\newcommand{\X}{\bb{X}}
\renewcommand{\a}{a}

\newcommand{\marked}{\scr{M}}
\newcommand{\face}{\tau}
\newcommand{\cell}{\tau}
\newcommand{\edge}{\sigma}

\newcommand{\eff}{f}

\newcommand{\up}{\U}
\newcommand{\upp}{\U_\ast}
\newcommand{\ep}{e_\mesh}
\newcommand{\epp}{e_{\mesh_\ast}}

\newcommand{\Ip}{I_\mesh^0}

\newcommand{\Ho}{H_0^2}
\renewcommand{\H}{H^2}

\newcommand{\mesh}{P}
\newcommand{\bedges}{\cal{G}}
\newcommand{\edges}{\cal{E}}
\newcommand{\Xp}{\bb{X}_\mesh}
\newcommand{\Xpp}{\bb{X}_{\mesh_\ast}}
\newcommand{\ap}{a_\mesh}
\newcommand{\app}{a_{\mesh_\ast}}
\newcommand{\Ep}{\scr{E}_\mesh}

\newcommand{\Op}{\cal{L}}
\newcommand{\Pip}{\Pi_\mesh}
\newcommand{\Pipp}{\Pi_{\mesh_\ast}}
\newcommand{\ellf}{\ell_f}
\newcommand{\osc}{\rm{osc}}
\newcommand{\etap}{\eta_\mesh}
\newcommand{\etapp}{\eta_{\mesh_\ast}}
\newcommand{\Ccoer}{C_\rm{coer}}
\newcommand{\Ccont}{C_\rm{cont}}

\newcommand{\cshape}{c_\rm{shape}}
\newcommand{\Crel}{C_\rm{Rel}}
\newcommand{\Cdrel}{C_\rm{dRel}}
\newcommand{\Ceff}{C_\rm{eff}}
\newcommand{\Cest}{C_\rm{est}}

\newcommand{\Clip}{C_\rm{lip}}
\newcommand{\Ccomp}{C_\rm{comp}}

\newcommand{\ctrace}{d_0}
\newcommand{\cinv}{d_1}
\newcommand{\cdtrace}{d_2}
\renewcommand{\c}{d_3}

\newcommand{\Cperp}{C_{\perp}}
\newcommand{\cproj}{c_1}

\title{A quasi-optimal adaptive spline-based finite element method for the bi-Laplace operator using Nitsche's method}
\author{Ibrahim Al Balushi}
\institution{McGill University}
\date{\today}

\begin{document}
\maketitle
\begin{abstract}
We establish the convergence of an adaptive spline-based finite element method of a fourth order elliptic problem with weakly-imposed Dirichlet boundary conditions using polynomial B-splines.
\end{abstract}
\section{Introduction}
Standard finite element methods (FEM) are based on triangular mesh partitions which have proven to be very robust at discretizing domains with complex geometry and are well-suited to problems requiring $H^1$ conforming shape functions. Higher degrees of smoothness across the element inferfaces is however much more involved. In recent years, with the emergence of \emph{isogeometric analysis} (IGA); see Hughes et al \cite{hughes2005isogeometric}, much attention has been directed at polynomial spline-based methods. Motivation began with the desire to integrate the CAD and analysis stages of design. As an immediate bonus, ploynomial spline-based meshes makes it easy to construct arbitrarily high orders of smoothness due to the mesh recutangular structure. In addition, NURB curves are robust at capturing curved geometries without the accumilation interpolation errors arising from standard trinagular-based FEM meshes.
However there is a drawback of using smooth spline-based bases for there is difficulty in prescribing essential boundary conditions (BC). Unlike nodal-based finite elements, smooth polynomial splines arrising from B-splines or NURBS are typically non-interpolatory which makes prescriptions of Dirichlet boundary conditions challenging and lead to highly oscillatory errors near the boundary \cite{bazilevs2007weak}.
In an earlier paper by Nitsche \cite{nitsche1971va} a weaker prescription of the boundary conditions is carried where BC are incorprated in the variational form rather than imposing it directly onto the discrete space \cite{stenberg1995some}. This idea hass been recently applied to the bi-Laplace operator \cite{embar2010imposing} using spline-based bases. An initial a posteriori analysis with this framework has been carried in \cite{juntunen2009nitsche} where the reliability and efficiency estimates are derived for the Poisson problem. However, the estimates included weighted boundary terms with negative powers and relied on a \emph{saturation assumption}. Recently, the idea has been employed in the treatment of a fourth-order elliptic problem appearing in geophysical flows
\cite{kim2015b},\cite{al2018adaptivity} with the added improvement that terms with negative powers were shown to be irrelevant much like in the case of adaptive discontinuous Galerkin methods (ADFEM)\cite{bonito2010quasi}.
While the analyses of \cite{morin2000data},\cite{morin2002convergence} justifies the use of the saturation assumption using a local lower bound in the Poisson problem, no such estimate is yet available for its fourth-order counterpart. In this work we aim to remove the saturation assumption as well as provide a convergence proof standard in residual-based AFEM literature of \cite{cascon2008quasi}. Many of the ideas are borrowed from the treament of ADFEM methods in \cite{bonito2010quasi} highlighting the similarity in nature of both mehods, theoreticaly as well as numerically.

Let $\Omega$ be a bounded domain in $\bb{R}^2$ with polygonal boundary $\Gamma$.
For a source function $f\in L^2(\Omega)$ we consider the following homogenous Dirichlet boundary-valued problem
\begin{eqnarray}\label{eq:pde}
\Op\u(x):=\Lap^2\u(x)=f(x)&&\text{in}\ \Omega\\
u=\diff u/\diff\nu=0&& \text{on}\ \Gamma.\nonumber
\end{eqnarray}
The adaptive procedure iterates over the following modules
\begin{equation}\label{eq:afem}
\boxed{\ms{SOLVE}}\longrightarrow\boxed{\ms{ESTIMATE}}\longrightarrow\boxed{\ms{MARK}}\longrightarrow\boxed{\ms{REFINE}}
\end{equation}
The module $\textbf{SOLVE}$ computes a hierarchical polynomial B-spline (HB) approximation $\U$ of the solution $\u$ with respect to a hierarchical partition $\mesh$ of $\Omega$.
For the module \textbf{ESTIMATE}, we use a residual-based error estimator $\eta_\mesh$ derived from the a posteriori analysis in Section~\ref{sec:post}.
The module \textbf{MARK} follows the D\"olfer marking criterion of \cite{dorfler1996convergent}.
Finally, the module \textbf{REFINE} produces a new refined partition $\mesh_\ast$ satisfying certain geometric constraints to ensure sharp approximation.\\
\\
\subsection{Notation}
We begin by laying out the notational conventions and function space definitions used in this presentation.
Let $\mesh$ be a partition of domain $\Omega$ consisting of square cells $\cell$ following the structure described in \cite{}.
Denote the collection of all interior edges of cells $\cell\in\mesh$ by $\edges_\mesh$ and all those along the boundary $\Gamma$ are to be collected in $\bedges_\mesh$.
We assume that cells $\tau$ are open sets in $\Omega$ and that edges $\sigma$ do not contain the vertices of its affiliating cell.
Let $\diam(\omega)$ be the longest length within a Euclidian object $\omega$ and set $h_\face:=\diam(\face)$ and $h_\edge:=\diam(\edge)$. Then let the mesh-size $h_\mesh:=\max_{\face\in\mesh}h_\face$.
Define the boundary mesh-size function $h_\Gamma\in L^\infty(\Gamma)$ by
\begin{equation}
h_\Gamma(x)=\sum_{\sigma\in\bedges_\mesh}h_\sigma\ds{1}_\sigma(x),
\end{equation}
where the $\ds{1}_\sigma$ are the indicator functions on boundary edges.
We define the support extension for a cell $\tau \in \mesh$ by
\begin{equation}
\omega_\tau=\{\tau'\in\mesh:\supp\beta\cap \tau'\neq\emptyset\implies\supp\beta\cap \tau\neq\emptyset\},
\end{equation}
indicating the collection of all supports for basis function $\beta$'s whose supports intersect $\tau$.
Analogously, we denote the support extension for an edge $\sigma\in\edges_\mesh\cup\bedges_\mesh$ by
\begin{equation}
\omega_\sigma=\{\tau\in\mesh:\supp\beta\cap\tau\neq\emptyset\implies\supp\beta\cap \tau\neq\emptyset,\ \sigma\subset\diff\tau\}.
\end{equation}
Let $H^s(\Omega)$, $s>0$, be the fractional order Sobolev space equipped with the usual norm $\norm{\cdot}{H^s(\Omega)}$; see references \cite{adams1975sobolev},\cite{grisvard2011elliptic}.
Let $H^s_0(\Omega)$ be given as the closure of the test functions $C_c^\infty(\Omega)$ in $\norm{\cdot}{H^s(\Omega)}$.
The semi-norm $\semi{\cdot}{H^s(\Omega)}$ defines a full norm on $H^s_0(\Omega)$ by virtue of Poincar\'e's inequality. Moreover, the semi-norm $\norm{\Lap\cdot}{L^2(\Omega)}$ defines a norm on $H^2_0(\Omega)$.
Let
\begin{equation}
\bb{E}(\Omega)=\left\{\v\in H_0^2(\Omega):\cal{L}\v\in L^2(\Omega)\right\}.
\end{equation}
By $H^{-2}(\Omega)=(H^{2}(\Omega))'$ the dual of $H^{2}(\Omega)$ with the induced norm
\begin{equation}
\norm{F}{H^{-2}(\Omega)}=\sup_{\v\in H^2(\Omega)}\frac{\inner{F,\v}}{\norm{\v}{H^2(\Omega)}}.
\end{equation}
We will be making use of the following mesh-dependent (semi)norms on $H^2(\Omega)$ which we employ in Nitsche's discretization:
\begin{equation}
\norm{\v}{s,\mesh}^2=\sum_{\edge\in\bedges_\mesh}h_\sigma^{-2s}\norm{\v}{L^2(\edge)}^2,
\end{equation}
\begin{equation}\label{eq:meshnorm}
\enorm{\v}{\mesh}^2=\norm{\Lap\v}{L^2(\Omega)}^2+\gamma_1\norm{\v}{3/2,\mesh}^2+\gamma_2\left\norm{\textstyle\dfrac{\v}{\nu}\right}{1/2,\mesh}^2,
\end{equation}
with $\gamma_1$ and $\gamma_2$ are suitably large positive stabilization parameters.
Finally, we denote  $a\lapprox b$ to indicate $a\leq Cb$ for a constant $C>0$ assumed to be independent of any notable parameters unless otherwise stated.

\subsection{Problem setup}
The natural weak formulation to the PDE \eqref{eq:pde} reads
\begin{equation}\label{eq:cwp}
\text{Find}\ \u\in\Ho(\Omega)\ \text{such that}\ \a(\u,\v)=\ellf(\v)\ \text{for all}\ \v\in\Ho(\Omega),
\end{equation}
where $\a:\Ho(\Omega)\times\Ho(\Omega)\to\bb{R}$ is be the bilinear form $\a(\u,\v)=(\Lap\u,\Lap\v)_{L^2(\Omega)}$ and $\ellf(\v)=(f,\v)_{L^2(\Omega)}$.
The energy norm $\enorm{\cdot}{}:=\sqrt{\a(\cdot,\cdot)}\equiv\norm{\Lap\cdot}{L^2(\Omega)}$ is one for which the form $\a$ is continuous and coercive on $H^2_0(\Omega)$, with unit proportionality constants, and the existence of a unique solution is therefore ensured by Babuska-Lax-Milgram theorem.
The variational formulation \eqref{eq:cwp} is consistent with the PDE \eqref{eq:pde} under sufficient regularity considerations; if $\u\in \bb{E}(\Omega)$ satisfies \eqref{eq:cwp} then $\u$ satisfies \eqref{eq:pde} in the classical sense by virtue of the Du Bois-Reymond lemma.
The space of piecewise polynomials of degree $r\ge2$ defined on a partition $\mesh$ will be given by
\begin{equation}
\cal{P}^r_\mesh(\Omega)=\prod_{\tau\in\mesh}\bb{P}_r(\tau).
\end{equation}
Assuming we have at our disposal a polynomial B-spline space $\Xp\subset\cal{P}_\mesh^r(\Omega)\cap H^2_0(\Omega)$ then an immediate discrete problem reads
\begin{equation}\label{eq:cdp}
\text{Find}\ \U\in\Xp\ \text{such that}\ \a(\U,\V)=\ellf(\V)\ \text{for all}\ \V\in\Xp.
\end{equation}
The corresponding linear system is numerically stable and consistent with \eqref{eq:cwp} in the sense that $\a(\u,\V)=\ellf(\V)$ for every $V\in\Xp$ and therefore we are provided with Galerkin orthogonality:
\begin{equation}\label{eq:cgo}
\a(\u-\U,\V)=0\quad\forall\V\in\Xp.
\end{equation}
Moreover, the spline solution to \eqref{eq:cdp} will serve as an optimal approximation to $\u$ in $\Xp$ with respect to $\enorm{\cdot}{}$:
\begin{equation}
\label{eq:result:lem:ccl}
\enorm{\u-\U}{}\leq\inf_{\V\in\Xp}\enorm{\u-\V}{}.
\end{equation}
The discretization given in \eqref{eq:cdp} requires prescription of the essential boundary values into the discrete spline space $\Xp$, and as mentioned earlier, this poses difficulty when considering non-homogenous boundary conditions due to the non-iterpolatory nature of high-order smoothness B-splines.
Therefore from now on we will depart from a boundary-value conforming discretization and assume that the spline space $\Xp\subset\cal{P}_\mesh^r(\Omega)\cap H^2(\Omega)$ no longer satisfies the boundary conditions and instead impose them weakly.
In the previous work \cite{al2018adaptivity} the following mesh-dependent bilinear form $\ap:\Xp\times\Xp\to\bb{R}$ is used to formulate Nitsche's discretization:
\begin{equation}\label{eq:ndp}
\text{Find}\ \U\in\Xp\ \text{such that}\ \ap(\U,\V)=\ellf(\V)\ \text{for all}\ \V\in\Xp.
\end{equation}
where
\begin{equation}\label{eq:oldnitbilin}
\begin{split}
\ap(\U,\V)&=\a(\U,\V)-\int_\Gamma\left(\textstyle\Lap\U\dfrac{\V}{\nu}+\Lap\V\dfrac{\U}{\nu}\right)
+\gamma_1\int_\Gamma h_\Gamma^{-3}\U\V
\\
&+\int_\Gamma\left(\textstyle\dfrac{\Lap\U}{\nu}\V
+\dfrac{\Lap\V}{\nu}\U\right)
+\gamma_2\int_\Gamma h_\Gamma^{-1}\textstyle\dfrac{\U}{\nu}\dfrac{\V}{\nu}.
\end{split}
\end{equation}
The discrete problem of \eqref{eq:ndp} with bilinear form \eqref{eq:oldnitbilin} is consistent with its continuous counterpart \eqref{eq:cwp} and quasi-optimal a priori error estimates have been realized; see \cite{al2018adaptivity}.
Unfortuantely, much like the analysis carried in \cite{juntunen2009nitsche},\cite{al2018adaptivity}, all \emph{a posteriori} estimates relied on the artificial so-called saturation assumption.
Here we will consider a modified version of the bilinear form \eqref{eq:oldnitbilin} which extends the domain of $\ap$ to all of $H^2(\Omega)$. This will enable us to remove the saturation assumption while carrying complete convergence analysis, and in an upcoming publication, an optimality analysis. Moreover, for discrete arguments the new bilinear form reduces back to \eqref{eq:oldnitbilin} . This will however be at the expense of consistency where we will no longer have access to \eqref{eq:cgo}. It will be shown that this obstacle is manageable and all desired conclusions will be met at the price of more delicate treatment.\\
\\
Let $\Pip:L^2(\Omega)\to\cal{P}^{r-2}_\mesh(\Omega)$ be the $L^2$-orthogonal projection operator given by
\begin{equation}
\forall\v\in L^2(\Omega),\ \Pip\v\in\cal{P}^{r-2}_\mesh(\Omega)
\ \text{such that}\
\int_\Omega\Pi_\mesh\v q=\int_\Omega \v q\quad\forall q\in\cal{P}_\mesh^{r-2}(\Omega).
\end{equation}
Instead of \eqref{eq:oldnitbilin} we consider the bilinear form $\ap:H^2(\Omega)\times H^2(\Omega)\to\mathbb{R}$
\begin{equation}\label{eq:nbf}
\begin{split}
\ap(\u,\v)&=\a(\u,\v)
-\int_\Gamma\left(\textstyle\Pip(\Lap\u)\dfrac{\v}{\nu}+\Pip(\Lap\v)\dfrac{\u}{\nu}\right)
+\gamma_1\int_\Gamma h_\Gamma^{-3}\u\v
\\
&+\int_\Gamma\left(\textstyle\dfrac{\Pip(\Lap\u)}{\nu}\v
+\dfrac{\Pip(\Lap\v)}{\nu}\u\right)
+\gamma_2\int_\Gamma h_\Gamma^{-1}\textstyle\dfrac{\u}{\nu}\dfrac{\v}{\nu}.
\end{split}
\end{equation}
The problem we will consider will read as \eqref{eq:ndp} but now with $\ap$ defined by \eqref{eq:nbf}.
To simplify notation we define
\begin{equation}\nonumber
\lambda_\mesh(u,v):=\int_\Gamma\left({\textstyle\dfrac{\Pip(\Lap\u)}{\nu}\v-\Pip(\Lap\u)\dfrac{\v}{\nu}}\right),
\quad
\lambda_\mesh^\ast(u,v):=\int_\Gamma\left(\textstyle\u\dfrac{\Pip(\Lap\v)}{\nu}-\dfrac{\u}{\nu}\Pip(\Lap\v)\right),
\end{equation}
\begin{equation}
\Sigma_\mesh(\u,\v):=\gamma_1\int_\Gamma h_\Gamma^{-3}\u\v+\gamma_2\int_\Gamma h_\Gamma^{-1}\textstyle\dfrac{\u}{\nu}\dfrac{\v}{\nu}.
\end{equation}
%
The solution $\u$ to \eqref{eq:cwp} does not satisfy the modified problem \eqref{eq:ndp}.
To quantify the inconsistency for $\u\in\bb{E}(\Omega)$, let $\Ep(\u)\in H^{-2}(\Omega)$ be given by
\begin{equation}\label{eq:it}
\inner{\Ep(\u),\v}=\int_\Gamma\left(\textstyle\dfrac{\Pip(\Lap\u)}{\nu}-\dfrac{\Lap\u}{\nu}\right)\v-
\int_\Gamma\left(\textstyle\Pip(\Lap\u)-\Lap\u\right)\textstyle\dfrac{\v}{\nu},
\quad\v\in H^2(\Omega).
\end{equation}
\begin{lemma}[Inconsistency]\label{lem:ac}
If $\u\in\bb{E}(\Omega)$ is the solution to \eqref{eq:cwp} then
\begin{equation}\label{eq:res:lem:ac}
\ap(\u,\v)=\ellf(\v)+\inner{\Ep(\u),\v}\quad\forall\v\in H^2(\Omega).
\end{equation}
\end{lemma}
\begin{proof}
Integrate by parts to get
\begin{equation}
\begin{split}
\ap(\u,\v)-\ellf(\v)&=\int_\Omega\left(\Op\u-f\right)\v
+\int_\Gamma{\Lap\u\textstyle\dfrac{\v}{\nu}}-\int_\Gamma{\textstyle\dfrac{\Lap\u}{\nu}}\v\\
&+\int_\Gamma{\textstyle\dfrac{\Pip(\Lap\v)}{\nu}}\u-\int_\Gamma\Pip(\Lap\v){\textstyle\dfrac{\u}{\nu}}
-\int_\Gamma{\textstyle\Pip(\Lap\u)\dfrac{\v}{\nu}}+\int_\Gamma\textstyle\dfrac{\Pip(\Lap\u)}{\nu},
\end{split}
\end{equation}
and
\begin{equation}
\int_\Omega\left(\Op\u-f\right)\v=\int_\Gamma{\textstyle\dfrac{\Pip(\Lap\v)}{\nu}}\u=\int_\Gamma\Pip(\Lap\v){\textstyle\dfrac{\u}{\nu}}=0\quad\forall\v\in H^2_0(\Omega),
\end{equation}
since $\u$ satisies the boundary valued differential equation \eqref{eq:pde}.
\end{proof}
\begin{remark}
It will be assumed from now on that the argument $\u\in\bb{E}(\Omega)$ in \eqref{eq:res:lem:ac} will aways be the continuous solution to \eqref{eq:cwp} and therefore we will drop the $(\u)$ from $\Ep(\u)$.
\end{remark}
\begin{remark}
Noting that $\Ho(\Omega)$ is in the kernel of $\Ep$, we see from \eqref{eq:res:lem:ac} that $\ap$ reduces to $\a$ and the discrete formulation \eqref{eq:ndp} is in fact consistent with \eqref{eq:cwp} whenever test functions $\v$ satisfy the boundary conditions.
\end{remark}
\begin{lemma}\label{lem:L2estimates}
Let $\mesh$ be an admissible partition, let $\tau\in\mesh$ and let $\sigma\in\bedges_\mesh$ with $\sigma\subset\diff\tau$.
The projection operator $\Pip$ satisfies the following stability estimates:
\begin{equation}\label{eq:stability:lem:L2estimates}
\norm{\Pip\v}{L^2(\Omega)}\leq\norm{\v}{L^2(\Omega)},
\end{equation}
and
\begin{equation}\label{eq:inverse:lem:L2estimates}
\left\norm{\Pip\v\right}{L^2(\sigma)}\leq\c h_\sigma^{-1/2}\norm{\v}{L^2(\tau)}
\quad\text{and}\quad
\left\norm{\textstyle\dfrac{(\Pip\v)}{\n_\sigma}\right}{L^2(\sigma)}\leq\c h_\sigma^{-3/2}\norm{\v}{L^2(\tau)},
\end{equation}
holding for every $\v\in L^2(\Omega)$.
\end{lemma}
\begin{proof}
The stability estimate \eqref{eq:stability:lem:L2estimates} follows from orthogonality of the residual $\v-\Pip\v$ to $\Pip\v$.
To establish \eqref{eq:inverse:lem:L2estimates}, we will only prove the second one, as the first estimate follows similarly.
Let $\v\in L^2(\tau)$.
In view of Lemma \ref{lem:ie} and stability \eqref{eq:stability:lem:L2estimates}
\begin{equation}
\begin{split}
\left\norm{\textstyle\dfrac{(\Pip\v)}{\n_\sigma}\right}{L^2(\sigma)}^2&
\leq\cinv h_\sigma^{-2}\norm{\Pip\v}{L^2(\sigma)}^2\\
&\leq\cdtrace\cinv h_\sigma^{-3}\norm{\Pip\v}{L^2(\tau)}^2
\leq\cdtrace\cinv h_\sigma^{-3}\norm{\v}{L^2(\tau)}^2.
\end{split}
\end{equation}
\end{proof}
We will assess the inconsistency and show that the formulation \eqref{eq:ndp} is in fact consistent asymptotically.
For this we will need some approximation tools.
\begin{lemma}\label{lem:L2error}
Let $\mesh$ be an admissible partition, let $\tau\in\mesh$ and let $\sigma\in\bedges_\mesh$ with $\sigma\subset\diff\tau$.
For a constant $\cproj>0$, depending only on $\cshape$,
if $0\leq t\leq s\leq r-1$ then
\begin{equation}\label{eq:result1:lem:L2error}
\semi{\v-\Pip(\v)}{H^t(\tau)}\leq \cproj h_\tau^{s-t}\semi{\v}{H^s(\tau)},
\end{equation}
and
\begin{equation}\label{eq:result2:lem:L2error}
\norm{\v-\Pip(\v)}{L^2(\sigma)}\leq\cproj h_\sigma^{s-1/2}\semi{\v}{H^s(\tau)},
\end{equation}
holding for every $\v\in\H(\Omega)$.
\end{lemma}
\begin{proof}
Let $1\leq t\leq s\leq r+1$ and let $\v\in\H(\Omega)$.
Let $\rho\in\bb{P}_r(\tau)$.
\begin{equation}
\begin{split}
\semi{\v-\Pip\v}{H^t(\tau)}&\leq\semi{\v-\rho}{H^t(\tau)}+\semi{\Pip(\rho-\v)}{H^t(\tau)}\\
&\leq\semi{\v-\rho}{H^t(\tau)}+\cinv h_\tau^{-t}\norm{\Pip(\rho-\v)}{L^2(\tau)}
\end{split}
\end{equation}
with the classical Bramble-Hilbert lemma we arrive at \eqref{eq:result1:lem:L2error} with $\cproj= (1+\cinv)c_\rm{HB}$ with $c_\rm{HB}>0$ is the proportionality constant of Bramble-Hilbert lemma.
Now in view of \eqref{eq:GeneralTrace}
\begin{equation}
\begin{split}
\norm{\v-\Pip\v}{L^2(\sigma)}^2\leq&\ctrace\left(h_\sigma^{-1}\norm{\v-\Pip\v}{L^2(\tau)}^2+h_\sigma\semi{\v-\Pi\v}{H^1(\tau)}^2\right)\\
&\ctrace\cproj\left(h_\sigma^{-1}h_\tau^{2s}\semi{\v}{H^s(\tau)}^2+
h_\sigma h_\tau^{2s-2}\semi{\v}{H^s(\tau)}^2\right)\\
&\leq \ctrace\cproj h_\sigma^{2s-1}\semi{\v}{H^s(\tau)}^2.
\end{split}
\end{equation}
\end{proof}
\begin{lemma}[Asymptotic consistency]\label{lem:AsymptoticConsistency}
If $\u\in\bb{E}(\Omega)$ is the solution to \eqref{eq:cwp} for which $\Lap\u\in H^s(\Omega)$, $s>0$, then for $\v\in H^2(\Omega)$,
\begin{equation}\label{eq:Resut:lem:AsymptoticConsistency}
\inner{\Ep,\v}\leq\cproj h_\mesh^s\left\norm{\Lap\u\right}{H^s(\Omega)}
\left(\norm{\v}{3/2,\mesh}+\left\norm{\textstyle\dfrac{\v}{\nu}\right}{1/2,\mesh}\right).
\end{equation}
\end{lemma}
\begin{proof}
\begin{equation}
\begin{split}
\inner{\Ep,\v}&=\int_\Gamma\left(\textstyle\dfrac{\Pip(\Lap\u)}{\nu}-\dfrac{\Lap\u}{\nu}\right)\v
-\int_\Gamma\left(\Pip(\Lap\u)-\Lap\u\right)\textstyle\dfrac{\v}{\nu}\\
&\leq\sum_{\sigma\in\bedges}\left\norm{\textstyle\dfrac{}{\n_\sigma}\left(\Pip(\Lap\u)-\Lap\u\right)\right}{L^2(\sigma)}\norm{\v}{L^2(\sigma)}
+\sum_{\sigma\in\bedges}\norm{\Pip(\Lap\u)-\Lap\u}{L^2(\sigma)}\left\norm{\textstyle\dfrac{\v}{\n_\sigma}\right}{L^2(\sigma)}
\end{split}
\end{equation}
In view of the projection error analysis of Lemma \eqref{lem:L2error}
\begin{equation}
\norm{\Pip(\Lap\u)-\Lap\u}{L^2(\sigma)}\leq\cproj h_\sigma^{s-1/2}\norm{\Lap\u}{H^s(\Omega)}
\end{equation}
and
\begin{equation}
\left\norm{\textstyle\dfrac{}{\n_\sigma}\left(\Pip(\Lap\u)-\Lap\u\right)\right}{L^2(\sigma)}\leq \cproj h_\sigma^{s-3/2}\norm{\Lap\u}{H^s(\Omega)}
\end{equation}
which leads us to the desired estimate.
\end{proof}
%
\subsection{The adaptive method}
We now recall the modules \textbf{SOLVE}, \textbf{ESTIMATE},
\textbf{MARK} and \textbf{REFINE}. A thorough discussion has already been carried in \cite{} with some minor differences.
\subsubsection*{The module SOLVE}
The discrete problem reads
\begin{equation}\label{eq:dnp}
\U=\ms{SOLVE}[\mesh,f]:\quad\text{Find}\ \U\in\Xp\ \text{such that}\ \ap(\U,\V)=\ellf(\V)\ \text{for all}\ \V\in\Xp.
\end{equation}
The stability of the problem will be addressed in Lemma \ref{lem:csvnbf} where we show that the bilinear form is coercive for large enough stabilization parameters $\gamma_1$ and $\gamma_2$. 
In view of the inconsistency \eqref{eq:res:lem:ac} we are left with partial Galerkin orthogonality:
\begin{equation}\label{eq:PartialGalerkin}
\ap(\u-\U,\V)=0\quad\forall\V\in\Xp\cap H^2_0(\Omega).
\end{equation}
\subsubsection*{The module ESTIMATE}
For a continuous function $\v$ we define the jump operator across interface $\sigma$.
\begin{equation}
\jump{\v}{\sigma}=\lim_{t\to0}[\v(x+t\sigma)-\v(x-tx)],\quad x\in\sigma.
\end{equation}
The adaptive refinement procedure of method \eqref{eq:afem} will aim to reduce the error estimations instructed by the cell-wise error indicators: for $\tau\in\mesh$
\begin{equation}\label{eq:indicator}
\eta_\mesh^2(\V,\face)=h_\face^4\norm{f-\Op\V}{L^2(\tau)}^2
+\sum_{\sigma\subset\diff\tau}
\left(\textstyle h_\sigma^{3}\left\norm{\jump{\dfrac{\Lap\V}{\n_\sigma}}{\sigma}\right}{L^2(\sigma)}^2
+h_\sigma\norm{\jump{\Lap\V}{\sigma}}{L^2(\sigma)}^2\right)
\end{equation}
We can define the indicators on subsets of $\Omega$ via:
\begin{equation}\label{eq:estimator}
\eta_\mesh^2(\V,\omega)=\sum_{\tau\in \mesh:\tau\subset\omega}\eta_\mesh^2(\V,\face),\quad\omega\subseteq\Omega
\end{equation}
To each cell $\tau$ in mesh $\mesh$ the error indicators \eqref{eq:indicator} will assign error estimations:
\begin{equation}
\{\eta_\tau:\tau\in\mesh\}=\ms{ESTIMATE}[\U,\mesh]:\quad\eta_\tau:=\eta_\mesh(\U,\U)
\end{equation}
We define data \emph{oscillation}
\begin{equation}
\osc_\mesh^2(\eff,\omega)=\sum_{\tau\subset\omega}h_\tau^4\norm{\eff-\Pip\eff}{L^2(\tau)}^2.
\end{equation}
\begin{remark}
Estimator dominance over oscialltion
\begin{equation}\label{eq:EstimatorDominance}
\osc_\mesh(\eff,\Omega)\leq\eta_\mesh(\U,\Omega)
\end{equation}
Estimator and oscillation monotonicity
\begin{equation}
\osc_{\mesh_\ast}(\eff,\Omega)\leq\osc_{\mesh}(\eff,\Omega),\quad
\eta_{\mesh_\ast}(\U_\ast,\Omega)\leq\eta_\mesh(\U,\Omega).
\end{equation}
\end{remark}
\subsubsection*{The module MARK}
We follow the Dorlfer marking strategy \cite{dorfler1996convergent}: For $0<\theta\leq1$,
\begin{equation}\label{markingstrategy}
\text{Find minimal spline set}\ \marked:\quad\sum_{\cell\in\marked}\eta^2_\mesh(\U,\cell)\ge\theta\sum_{\tau\in\mesh}\eta_\mesh^2(\U,\cell).
\end{equation}
To ensure minimal cardinality of $\cal{M}$ in the marking strategy one typically undergoes QuickSort which has an average complexity of $\cal{O}(n\log n)$ to produce the indexing set $J$.
\subsubsection*{The module REFINE}
Here we provide the important properties of \textbf{REFINE} which are needed in subsequent analyses and refer the reader to \cite{buffa2016adaptive},\cite{} for a detailed description. Procedure \textbf{REFINE} will ensure that for a constant $\cshape>0$, depending only on the polynomial degree of the spline space, all considered partitions therefore will satisfy the shape-regularity constraints:
\begin{eqnarray}\label{eq:sr}
\nonumber\sup_{\mesh\in\scr{P}}\max_{\tau\in\mesh}\#\left\{\tau\in\mesh:\tau\in\omega_\tau\right\}\leq\cshape&&\text{(finite-intersection property)},\\
\sup_{\mesh\in\scr{P}}\max_{\tau\in\mesh}\frac{\diam(\omega_\tau)}{h_\tau}\leq \cshape
&&(\text{graded}).
\end{eqnarray}
For any two partitions $\mesh_1,\mesh_2\in\scr{P}$ there exists a common admissible partition in $\scr{P}$, called the \emph{overlay} and denoted by $\mesh_1\oplus \mesh_2$, such that 
\begin{equation}\label{eq:MeshOverelay}
\#(\mesh_1\oplus\mesh_2)\leq\#\mesh_1+\#\mesh_2-\#\mesh_0.
\end{equation}
Moreover, shown in \cite{buffa2016complexity},
if the sequence $\{\mesh_\ell\}_{\ell\ge1}$ is obtained by repeating the step $\mesh_{\ell+1}:=\ms{ REFINE}\,[\mesh_\ell,\scr{M}_\ell]$ with $\scr{M}_\ell$ any subset of $\mesh_\ell$, then for $k\ge1$ we have that
\begin{equation}\label{eq:MarkingComplexity}
\#P_k-\#P_\ell\leq \Lambda\sum_{\ell=1}^{k}\#\scr{M}_\ell.
\end{equation}
where $\Lambda>0$ which will depend on the polynomial degree $r$.
\section{A priori analysis for Nitsche's formulation}
In what follows we show the proposed discrete problem admits an a priori estimate. This will be immediate from upon estabishing that mesh-dependent bilinear form is bounded and coercive for sufficiently large stabilization parameters $\gamma_1$ and $\gamma_2$ with respect to mesh-dependent norm \eqref{eq:meshnorm}. 
\begin{lemma}[Continuity of $\ap$]\label{lem:ctynbf}
Let $\gamma_1,\gamma_2>0$ be given.
We have
\begin{equation}\label{eq:Continuity}
|\ap(\u,\v)|\leq\Ccont\enorm{\u}{\mesh}\enorm{\v}{\mesh}\quad\u,\v\in H^2(\Omega),
\end{equation}
with a constant $\Ccont>0$ independent of $\mesh$.
\end{lemma}
\begin{proof}
We begin with the interior integrals;
\begin{equation}
\ap(\u,\v)\leq \norm{\Lap\u}{L^2(\Omega)}\norm{\Lap\v}{L^2(\Omega)}.
\end{equation}
As for the boundary terms,
\begin{equation}
\lambda_\mesh^\ast(\u,\v)\leq\norm{\u}{L^2(\Gamma)}\left\norm{\textstyle\dfrac{\Pi(\Lap\v)}{\nu}\right}{L^2(\Gamma)}
+\left\norm{\textstyle\dfrac{\u}{\nu}\right}{L^2(\Gamma)}\left\norm{\Pi(\Lap\v)\right}{L^2(\Gamma)}
\end{equation}
\begin{equation}
\begin{split}
\lambda_\mesh^\ast(\u,\v)&\lapprox\left\norm{h_\Gamma^{-3/2}\u\right}{L^2(\Gamma)}\left\norm{\Lap\u\right}{L^2(\Omega)}
+\left\norm{\textstyle h_\Gamma^{-1/2}\dfrac{\u}{\nu}\right}{L^2(\Gamma)}\left\norm{\Lap\v\right}{L^2(\Omega)}\\
&\lapprox \left(\norm{\u}{3/2,\mesh}+\left\norm{\textstyle\dfrac{\u}{\nu}\right}{1/2,h}\right)\left\norm{\Lap\v\right}{L^2(\Omega)}
\end{split}
\end{equation}
Similarily,
\begin{equation}
\lambda_\mesh(\u,\v)\leq \left\norm{\Lap\u\right}{L^2(\Omega)}\left(\norm{\v}{3/2,\mesh}+\left\norm{\textstyle\dfrac{\v}{\nu}\right}{1/2,\mesh}\right).
\end{equation}
The stabilization terms are similarly controlled
\begin{equation}
\Sigma_\mesh(\u,\v)\leq\gamma_1\left\norm{\u\right}{3/2,\mesh}\left\norm{\v\right}{3/2,\mesh}+\gamma_2\left\norm{\textstyle\dfrac{\u}{\nu}\right}{1/2,\mesh}\left\norm{\textstyle\dfrac{\v}{\nu}\right}{1/2,\mesh}
\end{equation}
\end{proof}
%
%
\begin{lemma}[Coercivity of $\ap$]\label{lem:csvnbf}
For suitably large stabilization parameters $\gamma_1$ and $\gamma_2$, there exists a constant $\Ccoer>0$ such that
\begin{equation}\label{eq:Coercivity}
\Ccoer\enorm{\v}{\mesh}^2\leq\ap(\v,\v)\quad\forall\v\in H^2(\Omega).
\end{equation}
\end{lemma}
\begin{proof}
For $\delta_1,\,\delta_2>0$ we use Young's inequality to write
\begin{equation}
\begin{split}
\lambda_\mesh(\v,\v)+\lambda_\mesh^\ast(\v,\v)&\ge-\frac{1}{\delta_1}\norm{\v}{L^2(\Gamma)}^2
-\delta_1\left\norm{\textstyle\dfrac{\Pip(\Lap\v)}{\nu}\right}{L^2(\Gamma)}^2
-\frac{1}{\delta_2}\left\norm{\textstyle\dfrac{\v}{\nu}\right}{L^2(\Gamma)}^2
-\delta_2\left\norm{\Pip(\Lap\v)\right}{L^2(\Gamma)}^2.
\end{split}
\end{equation}
Together with the interior terms we have
\begin{equation}
\begin{split}
\ap(\v,\v)&\ge\norm{\Lap\v}{L^2(\Omega)}^2-\delta_1\left\norm{\textstyle\dfrac{\Pip(\Lap\v)}{\nu}\right}{L^2(\Gamma)}^2-\delta_2\left\norm{\Pip(\Lap\v)\right}{L^2(\Gamma)}^2\\
&+\left(1-\frac{1}{\gamma_1}-\frac{1}{\delta_1\gamma_1}\right)\gamma_1\left\norm{\psi\right}{3/2,\mesh}^2
+\left(1-\frac{1}{\delta_2\gamma_2}\right)\gamma_2\left\norm{\textstyle\dfrac{\v}{\nu}\right}{1/2,\mesh}^2
\end{split}
\end{equation}
With inverse estimates \eqref{eq:inverse:lem:L2estimates} 
\begin{equation}
\begin{split}
\ap(\v,\v)&\ge\left(1-\delta_1C\max_{\sigma\in\bedges_\mesh}h_\sigma^{-3/2}-\delta_2C\max_{\sigma\in\bedges_\mesh}h_\sigma^{-1/2}\right)\norm{\Lap\v}{L^2(\Omega)}^2\\
&+\left(1-\frac{1}{\gamma_1}-\frac{1}{\delta_1\gamma_1}\right)\gamma_1\left\norm{\v\right}{3/2,\mesh}^2
+\left(1-\frac{1}{\delta_2\gamma_2}\right)\gamma_2\left\norm{\textstyle\dfrac{\v}{\nu}\right}{1/2,\mesh}^2,
\end{split}
\end{equation}
For sufficiently small $\delta_1$ and $\delta_2$, pick $\gamma_1$ and $\gamma_2$ sufficiently large to yield the desired result. 
\end{proof}
Continuity and coercivity of the bilinear form ensures a unique solution $\U$ to the discrete problem \eqref{eq:dnp} which admits the following \emph{a priori} estimate.
\begin{lemma}[A priori error estimate for Nitsche's forumation]
Let $\u\in H^2_0(\Omega)$ be a solution to \eqref{eq:cwp} with $\Lap\u\in H^s(\Omega)$ with $s>\frac{3}{2}$. For stabilization paremeters $\gamma_1,\gamma_2>0$ satisfying the hypothesis of Lemma \ref{lem:csvnbf}, 
\begin{equation}
\enorm{\u-\U}{\mesh}\leq\left(1+\frac{\Ccont}{\Ccoer}\right)\inf_{\V\in\Xp}\enorm{\u-\V}{\mesh}+\frac{1}{\Ccoer}\norm{\Ep}{H^{-2}(\Omega)}.
\end{equation}
\end{lemma}
\begin{proof}
From
\begin{equation}
\enorm{\u-\U}{\mesh}\leq\enorm{\u-\V}{\mesh}+\enorm{\V-\U}{\mesh}
\end{equation}
we will estimate $\enorm{\V-\U}{\mesh}$. Let $\W=\V-\U$,
\begin{equation}
\begin{split}
\Ccoer\enorm{\V-\U}{\mesh}^2&\leq\ap(\V-\u,\W)+\ap(\u-\U,\W)\\
&\leq\Ccont\enorm{\u-\V}{\mesh}\enorm{\W}{\mesh}+|\!\inner{\Ep,\W}\!|\\
&\leq\Ccont\enorm{\u-\V}{\mesh}\enorm{\W}{\mesh}+\norm{\Ep}{H^{-2}(\Omega)}\enorm{\W}{\mesh}
\end{split}
\end{equation}
which makes
\begin{equation}
\enorm{\u-\U}{\mesh}\leq\left(1+\textstyle\frac{\Ccont}{\Ccoer}\right)\enorm{\u-\V}{\mesh}+\textstyle\frac{1}{\Ccoer}\norm{\Ep}{H^{-2}(\Omega)}
\end{equation}
The quantity $\norm{\Ep}{H^{-2}(\Omega)}$ is finite by Lemma \ref{lem:AsymptoticConsistency}.
\end{proof}

\section{A posteriori estimates}\label{sec:post}	
In this section we will derive the a posteriori error estimates for \eqref{eq:dnp} which will yield convergence of spline solutions generated by the iterative procedure \eqref{eq:afem} to the the weak solution $\u$ of \eqref{eq:cwp}. Contrary to \cite{}, estimating the residual $\scr{R}_\mesh=\eff-\cal{L}\U$ is not possible due to the inconsisency. The estimate \eqref{eq:Resut:lem:AsymptoticConsistency} assumes $\Lap\u\in H^s(\Omega)$ for $s>\frac{3}{2}$ which is too high. A more delicate treatement is needed in which $\ap(e,e)$ will be approximated directly.
We will need some approximation tools and estimates, discussed in greater detail in \cite{} with reference to ~\cite{speleers2017hierarchical},\cite{speleers2016effortless},\cite{bazilevs2006isogeometric}, for spline spaces $\Xp\subset H^2_0(\Omega)$. We will use the same quasi-interpolation projections onto $\Xp\cap H^2_0(\Omega)$.
\subsection{Approximation in $\Xp$}
Recall the general trace theorem \cite{adams1975sobolev},\cite{grisvard2011elliptic} for cells $\tau\in\mesh$ and edges $\sigma\in\bedges_\mesh$ with $\sigma\subset\diff\tau$.
For a constant $\ctrace>0$
\begin{equation}\label{eq:GeneralTrace}
\norm{\v}{L^2(\sigma)}^2\leq \ctrace\left(h_\sigma^{-1}\norm{\v}{L^2(\tau)}^2+h_\sigma\norm{\grad\v}{L^2(\tau)}^2\right)\quad\forall\v\in H^1(\Omega).
\end{equation}
\begin{lemma}[Auxiliary discrete estimate]\label{lem:ie}
Let $\face\in\mesh$.
Then for $\cinv>0$, depending only on polynomial degree $r$, for $0\leq s\leq t\leq r+1$ we have
\begin{equation}\label{eq:inve:lem:ie}
\semi{\V}{H^t(\tau)}\leq\cinv h_\tau^{s-t}\semi{\V}{H^s(\tau)}\quad\forall\V\in\bb{P}_r(\tau),
\end{equation}
and if $\edge\subset\diff\face$, for a constant $\cdtrace>0$ we have
\begin{equation}\label{eq:dtrace:lem:ie}
\norm{\V}{L^2(\sigma)}\leq\cdtrace h_\sigma^{-1/2}\norm{\V}{L^2(\tau)}\quad\forall\V\in\bb{P}_r(\tau),
\end{equation}
where $\cdtrace:=\ctrace\max\{1,\cinv\}$.
\end{lemma}
\begin{remark}
The constants $\cinv,\ \ctrace,\ \cdtrace$ all depend on the polynomial degree and the reference cell or edge; $\hat{\tau}=[0,1]^2$ or $\hat{\sigma}=[0,1]$.
From now, for a simpler presentation of the analysis, we combined all these constants, and their powers into a unifying constant $c_\ast$
\end{remark}
We recall from \cite{}:.
%
\begin{lemma}[Quasi-interpolantion]
Let $\mesh$ be an admissible partition of $\Omega$. There exists a quasi-interpolantion operator $\Ip:L^2(\Omega)\to\Xp\cap H^2_0(\Omega)$ such that for every $\face\in\mesh$,
\begin{equation}
\norm{I^0_\mesh\v}{L^2(\cell)}\lapprox\cshape\norm{\v}{L^2(\omega_\cell)}\quad\forall\v\in L^2(\omega_\tau),
\end{equation}

\begin{equation}\label{eq:Quasi1}
\semi{\v-I^0_\mesh\v}{L^2(\cell)}\lapprox\cshape h_\cell^{2}\semi{\v}{H^2(\omega_\cell)}\quad\forall\v\in H_0^2(\omega_\tau),
\end{equation}
and for $k=0,1$
\begin{equation}\label{eq:Quasi2}
\forall\sigma\in\cal{E}_\mesh,\quad\semi{\v-\Ip \v}{H^k(\sigma)}\leq h_\sigma^{3/2-k}\semi{\v}{H^2(\omega_\sigma)}\quad\forall\v\in H_0^2(\omega_\sigma).
\end{equation}
\end{lemma}

Let $\Xp^0=\Xp\cap H_0^2(\Omega)$.
We characterize an orthogonal complement $\Xp^\perp$ to $\Xp^0$ using a projection operator $\pi^0_\mesh:\Xp\to\Xp^0$ defined by the linear problem
\begin{equation}
\pi_\mesh\V\in\Xp^0:\quad\ap(\W_0,\V-\pi^0_\mesh\V)=0\quad\forall\W_0\in\Xp^0.
\end{equation}
By setting $\pi_\mesh^\perp\V=\V-\pi_\mesh^0\V$ for any $\V\in\Xp$, we obtain a decompose for every finite-element spline
\begin{equation}
\V=\pi_\mesh^0\V+\pi_\mesh^\perp\V=:\V^0+\V^\perp\in\Xp^0\oplus\Xp^\perp\equiv\Xp
\end{equation}
with 
\begin{equation}\label{eq:OrthogonalDecomp}
\ap(\V^0,\W^\perp)=0
\end{equation}
for every pair $\V$ and $\W$.
We have the following result:
\begin{lemma}\label{lem:PerpNormEquiv}
Semi-norm $\norm{\cdot}{3/2,\mesh}+\left\norm{\dfrac{\,\cdot\,}{\nu}\right}{1/2,\mesh}$ defines a norm on $\Xp^\perp$. In particular, for a constant $\Cperp>0$
\begin{equation}\label{eq:lem:PerpNormEquiv}
\enorm{\V^\perp}{\mesh}\leq\Cperp\left(\norm{\V^\perp}{3/2,\mesh}+\left\norm{\dfrac{\V^\perp}{\nu}\right}{1/2,\mesh}\right)\quad\forall\V^\perp\in\Xp^\perp.\end{equation}
\end{lemma}
\begin{proof}
Let $
\cal{D}_\Gamma=\rm{Int}\bigg(\overline{\Omega\cap\bigcup_{\sigma\in\cal{G}_\mesh}\omega_\sigma}\bigg)$. If $\norm{\V^\perp}{3/2,\mesh}+\norm{\dfrac{\V^\perp}{\nu}}{1/2,\mesh}=0$ then $\V^\perp=\dfrac{\V^\perp}{\nu}\equiv0$ on $\cal{D}_\Gamma$ due to the finite-dimensionality of polynomial space $\Xp^\perp$. Necessarily we have $\V^\perp\equiv0$ everywhere; otherwise $\V^\perp\in\Xp^0$.
A more detailed treatment has already been carried in \cite{al2018adaptivity}.
\end{proof}
\begin{lemma}\label{lem:SolveConformPart}
Let $\U=\U^0+\U^\perp$ be the spline solution to \eqref{eq:dnp}
\begin{equation}\label{eq:lem:SolveConformPart}
\ap(\U^0,\V^0)=\ellf(\V^0)\quad\forall\V^0\in\Xp^0.
\end{equation}
\end{lemma}
\begin{proof}
We have by symmetry and \eqref{eq:OrthogonalDecomp}
\[
\ap(\U,\V)=\ap(\U^0,\V^0)+\ap(\U^\perp,\V^\perp)=\ellf(\V^0)+\ellf(\V^\perp)\quad\forall\V\in\Xp.
\]
since $\V\in\Xp$ is arbitrary we arrive at \eqref{eq:lem:SolveConformPart}.
\end{proof}

We prove that the proposed error estimator is reliable. 
The idea is to express $\ap(\e,\e)$ as a sum of two terms, the first quantifies the interior and edge jump residual terms, essentially capturing the spacial locations where the solution exhibits loss in regualrity, and the second term arrising from the formulation's inconsistency.

\begin{lemma}[Estimator reliability]\label{lem:EstimatorReliability}
Let $\mesh$ be a partition of $\Omega$ satisfying Conditions \eqref{eq:sr}.
The module $\ms{ESTIMATE}$ produces a posteriori error estimate $\eta_\mesh$ for the discrete error such that for a constants $C_\rm{rel,1},C_\rm{rel,2}>0$,
\begin{equation}\label{eq:res:lem:EstimatorReliability}
\begin{split}
\ap(\u-\U,\u-\U)&\leq C_\rm{rel,1}\eta_\mesh^2(\U,\Omega)
+C_\rm{rel,2}\left(\gamma_1\norm{\U}{3/2,\mesh}^2+\gamma_2\left\norm{\textstyle\dfrac{\U}{\nu}\right}{1/2,\mesh}^2\right),
\end{split}
\end{equation}
with constants depending only on $\cshape$.
\end{lemma}
\begin{proof}
Let $e=\u-\U$ and let $\v=\u-\U^0$ and we may write $e=\v-\U^\perp$. Since $\Ip \v\in\Xp^0$, Partial Galerkin orthogonality \eqref{eq:PartialGalerkin} implies $\ap(e,\Ip \v)=0$ and we have
\begin{equation}\label{maineq:lem:EstimatorReliability}
\ap(e,e)=\ap(e,\v-\Ip \v)-\ap(e,\U^\perp).
\end{equation}
The treatment of the term $\ap(e,\v-\Ip \v)$ is similar that in \cite{} except that now we have to control the additional boundary integrals.
\begin{equation}\label{eq1:lem:EstimatorReliability}
\begin{split}
|\ap(e,\v-\Ip \v)|\leq&\sum_{\tau\in\mesh}\norm{f-\Op\U}{L^2(\tau)}\norm{\v-\Ip \v}{L^2(\tau)}
+\sum_{\sigma\in\edges_\mesh}\left\norm{\jump{\textstyle\dfrac{\Lap\U}{\n_\sigma}}{\sigma}\right}{L^2(\sigma)}\left\norm{\textstyle(\v-\Ip \v)\right}{L^2(\sigma)}\\
&+\sum_{\sigma\in\edges_\mesh}\norm{\jump{\Lap\U}{\sigma}}{L^2(\sigma)}\left\norm{{\textstyle\dfrac{}{\n_\sigma}}(\v-\Ip \v)\right}{L^2(\sigma)}
+\left|\int_\Gamma{\textstyle\dfrac{\U}{\nu}}\Pi[\Lap(\v-\Ip \v)]\right|,\\
&+\left|\int_\Gamma\U{\textstyle\dfrac{\Pi[\Lap(\v-\Ip \v)]}{\nu}}\right|.
\end{split}
\end{equation}
For the boundary intergrals,
\begin{equation}
\begin{split}
\left|\int_\Gamma{\textstyle\dfrac{\U}{\nu}}\Pi[\Lap(\v-\Ip \v)]\right|
&\leq\sum_{\sigma\in\bedges_\mesh}\left\norm{\textstyle\dfrac{\U}{\n_\sigma}\right}{L^2(\sigma)}\norm{\Pi[\Lap(\v-\Ip \v)]}{L^2(\sigma)},\\
&\leq\cdtrace\sum_{\sigma\in\bedges_\mesh}\left\norm{\textstyle\dfrac{\U}{\n_\sigma}\right}{\sigma}h_\sigma^{-1/2}\norm{\Lap(\v-\Ip \v)}{L^2(\tau(\sigma))},\\
&\leq\cproj \cdtrace\bigg(\sum_{\sigma\bedges_\mesh}h_\sigma^{-1}\left\norm{\textstyle\dfrac{\U}{\n_\sigma}\right}{\sigma}^2\bigg)^{1/2}
\bigg(\sum_{\sigma\in\bedges_\mesh}\norm{\v}{H^2(\omega_\tau)}^2\bigg)^{1/2},
\end{split}
\end{equation}
where $\tau(\sigma)$ is the boundary adjacent cell with edge $\sigma$.
Similarly,
\begin{equation}
\begin{split}
\left|\int_\Gamma\U{\textstyle\dfrac{\Pi[\Lap(\v-\Ip \v)]}{\nu}}\right|
&\leq\sum_{\sigma\in\bedges_\mesh}\norm{\U}{L^2(\sigma)}\left\norm{\textstyle\dfrac{}{\n_\sigma}\Pi[\Lap(\v-\Ip \v)]\right}{L^2(\sigma)},\\
&\leq\cinv\cdtrace\sum_{\sigma\in\bedges_\mesh}\norm{\U}{L^2(\sigma)}h_\sigma^{-3/2}\norm{\Lap(\v-\Ip \v)}{L^2(\tau(\sigma))},\\
&\leq \cproj \cinv\cdtrace\bigg(\sum_{\sigma\in\bedges_\mesh}h_\sigma^{-3}\norm{\U}{L^2(\sigma)}^2\bigg)^{1/2}\bigg(\sum_{\sigma\in\bedges_\mesh}\norm{\v}{H^2(\tau(\sigma))}^2\bigg)^{1/2}.
\end{split}
\end{equation}
If $C_1=\cproj\cdtrace\max\{\cinv,1\}$,
\[
\cproj\cdtrace\left\{\cinv\norm{\U}{3/2,\mesh}+\left\norm{\textstyle\dfrac{\U}{\nu}\right}{1/2,\mesh}\right\}
\leq C_1\semi{\U^\perp}{\mesh}
\]
We define the interior residual terms $R_\tau=(f-\Op\U)|_\tau$ for every cell $\tau\in\mesh$ and edge jump terms $J_{\sigma.1}=\textstyle\jump{\dfrac{\Lap\U}{\n_\sigma}}{\sigma}$ and $J_{\sigma,2}=\jump{\Lap\U}{\sigma}$ across  each interior edge $\sigma$. 
We arrive at
\begin{equation}
\begin{split}
|\ap(e,\v-\Ip \v)|&\leq \cproj \left\{\bigg(\sum_{\tau\in\mesh}h_\tau^4\norm{R_\tau}{L^2(\tau)}^2\bigg)^{1/2}
+\bigg(\sum_{\sigma\in\edges_\mesh}h_\sigma^3\left\norm{J_{\sigma,1}\right}{L^2(\sigma)}^2\bigg)^{1/2}\right.\\
&\quad\left.+\bigg(\sum_{\sigma\in\edges_\mesh}h_\sigma\norm{J_{\sigma,2}}{L^2(\sigma)}^2\bigg)^{1/2}
\right\}\norm{\v}{H^2(\Omega)}+C_1\semi{\U^\perp}{P}\norm{\v}{H^2(\Omega)}
\end{split}
\end{equation}
Let 
\begin{equation}
\eta_\mesh(\Omega)=\bigg(\sum_{\tau\in\mesh}h_\tau^4\norm{R_\tau}{L^2(\tau)}^2\bigg)^{1/2}
+\bigg(\sum_{\sigma\in\edges_\mesh}h_\sigma^3\left\norm{J_{\sigma,1}\right}{L^2(\sigma)}^2\bigg)^{1/2}
+\bigg(\sum_{\sigma\in\edges_\mesh}h_\sigma\norm{J_{\sigma,2}}{L^2(\sigma)}^2\bigg)^{1/2}.
\end{equation}
To control the inconsistency term $\ap(e,\U^\perp)$, we employ Young's inequality and the norm equivalence from Lemma \ref{lem:PerpNormEquiv}
\begin{equation}\label{eq2:lem:EstimatorReliability}
\begin{split}
\ap(e,\U^\perp)&\leq\Ccont\enorm{e}{P}\enorm{\U^\perp}{P}
\leq\frac{\Ccont}{\Ccoer^{1/2}}\ap(e,e)^{1/2}\enorm{\U^\perp}{P},\\
&\leq\frac{\ap(e,e)}{4}+\frac{\Ccont^2}{\Ccoer}\enorm{\U^\perp}{P}^2
\leq\frac{\ap(e,e)}{4}+\frac{\Ccont^2}{\Ccoer}\Cperp^2\semi{\U^\perp}{P}^2.
\end{split}
\end{equation}
Let $C_2=\frac{\Ccont^2}{\Ccoer}\Cperp^2$.
Since $\v=e+\U^\perp$
\begin{equation}
\begin{split}
\norm{\v}{H^2(\Omega)}^2&\leq\Ccoer^{-1}\ap(e+\U^\perp,e+\U^\perp),\\
&=\Ccoer^{-1}\left(\ap(e,e)+2\ap(e,\U^\perp)+\ap(\U^\perp,\U^\perp)\right),\\
&\leq\Ccoer^{-1}\left(2\ap(e,e)+\left(1+2C_2\right)\semi{\U^\perp}{P}^2\right).
\end{split}
\end{equation}
Let $C_3^2=\Ccoer^{-1}\max\{2,(1+2C_2)\}$.
Summing up, applying Young's inequality with $\delta=1/2$,
\begin{equation}
\begin{split}
\frac{3}{4}\ap(e,e)&\leq\cproj\left(\eta_P(\Omega)+C_1\semi{\U^\perp}{P}\right)\norm{\v}{H^2(\Omega)}+C_2\semi{\U^\perp}{P}^2,\\
&\leq\cproj C_3\left(\eta_P(\Omega)+C_1\semi{\U^\perp}{P}\right)\left(\ap(e,e)+\semi{\U^\perp}{P}^2\right)^{1/2}+C_2\semi{\U^\perp}{P}^2,\\
&\leq C_3\left(\eta_P(\Omega)+C_1\semi{\U^\perp}{P}\right)^2
+\frac{1}{4}\left(\ap(e,e)+\semi{\U^\perp}{P}^2\right)+C_2\semi{\U^\perp}{P}^2,
\end{split}
\end{equation}
which makes for constants $C_\rm{rel,1}>0$ and $C_\rm{rel,2}>0$ depending on $C_1$, $C_2$ and $C_3$,
\begin{equation}
\begin{split}
\frac{1}{2}\ap(e,e)\leq&\frac{C_\rm{rel,1}}{2}\eta_P(\Omega)
+\frac{C_\rm{rel,2}}{2}\semi{\U^\perp}{P}^2.
\end{split}
\end{equation}


\end{proof}

The following lemma shows that the proposed estimator from Lemma \ref{lem:EstimatorReliability} is efficient in the sense that $\eta_\mesh$ is a sharp approximation to the error $\enorm{\u-\U}{\mesh}$ up to how well the partition resolves the source function $\eff$.
%
\begin{lemma}[Estimator Efficiency]\label{lem:EstimatorEfficiency}
Let $\mesh$ be a partition of $\Omega$ satisfying conditions \eqref{eq:sr}.
The module $\ms{ESTIMATE}$ produces a posteriori error estimate of the discrete solution error such that
\begin{equation}\label{eq:result:lem:EstimatorEfficiency}
\Ceff\,\eta_\mesh^2(\U,\Omega)\leq\enorm{\u-\U}{\mesh}^2+\osc_\mesh^2(\Omega).
\end{equation}
with constant $\Ceff$ depending only on $\cshape$.
\end{lemma}
In the following Lemma we show a local version of Lemma \ref{lem:EstimatorReliability}. While the result is not needed for convergence, it is required for quasi-optimality.
\begin{lemma}[Estimator discrete reliability]\label{lem:DiscreteEstimatorReliability}
Let $\mesh$ be a partition of $\Omega$ satisfying conditions \eqref{eq:sr} and let $\mesh_\ast=\ms{REFINE}\,[\mesh,R]$ for some refined set $R\subseteq\mesh$.
If $\U$ and $\U_\ast$ are the respective solutions to \eqref{eq:cdp} on $\mesh$ and $\mesh_\ast$, then for a constants $ C_\rm{dRel,1}, C_\rm{dRel,2}>0$, depending only on $\cshape$,
\begin{equation}\label{eq:result:lem:DiscreteEstimatorReliability}
\begin{split}
\enorm{\U^0_\ast-\U}{\mesh}^2&\leq C_\rm{dRel,1}\eta^2_\mesh(\U,\omega_{R_{\mesh\to\mesh_\ast}})
+C_\rm{dRel,2}\left(\gamma_1\norm{\U}{3/2,R}^2+\gamma_2\left\norm{\textstyle\dfrac{\U}{\nu}\right}{1/2,R}^2\right),
\end{split}
\end{equation}
where $\omega_{R_{\mesh\to\mesh_\ast}}$ is understood as the union of support extensions of refined cells from $\mesh$ to obtain $\mesh_\ast$.
\end{lemma}

\begin{proof}
In view of \eqref{eq:lem:SolveConformPart} and the nesting of spline spaces, $\ap(\U_\ast^0,\V^0)=\ellf(\V^0)$ holds if $\V^0\in\Xp^0$ from which we obtain $\ap(\U_\ast^0-\U,\V^0)=0$ for every $\V^0\in\Xp^0$.
Let $E_\ast^0=\U_\ast^0-\U_0$ and let $E_\ast=\U^0_\ast-\U\equiv E_\ast^0-\U^\perp$. Then for any $\V_0\in\Xp^0$ we write an analogous expression to \eqref{maineq:lem:EstimatorReliability}
\begin{equation}\label{maineq:lem:}
\ap(E_\ast,E_\ast)=\ap(E_\ast,E_\ast^0-\U^\perp)=\ap(E_\ast,E^0_\ast-\V^0)-\ap(E_\ast,\U^\perp)
\end{equation}
which we proceed to control in terms of the estimator.
For the first term, we form disconnected subdomains $\Omega_i\subseteq\Omega$, $i\in J$, each formed from the interior of connected union of cell support extensions. Set $\Omega_\ast=\cup_{\tau\in R_{\mesh\to\mesh_\ast}}\overline{\omega_\tau}$.
Then to each subdomain $\Omega_i$ we form a partition $\mesh_i=\{\face\in\mesh:\face\subset\Omega_i\}$, interior edges $\edges_i=\{\sigma\in\edges_\mesh:\sigma\subset\diff\tau,\ \tau\in\mesh_i\}$ and boundary edges $\bedges_i=\{\sigma\in\bedges_\mesh:\sigma\subset\diff\tau,\ \tau\in\mesh_i\}$, and a corresponding finite-element space $\X_i$. Let $I_i:L^2(\Omega_i)\to\X_i$ satisfy the local estimates \eqref{eq:Quasi1} and \eqref{eq:Quasi2}
Let $\V^0\in\Xp^0$ be an approximation of $E^0_\ast$ be given by
\begin{equation}
\V^0=E^0_\ast\ds{1}_{\Omega\backslash\Omega_\ast}+\sum_{i\in J}(I_i^0E^0_\ast)\cdot\ds{1}_{\Omega_i}.
\end{equation}
Then $E^0_\ast-\V^0\equiv0$ on $\Omega\backslash\Omega_\ast$. 
To localize the error on $\omega_{R_{P\to P_\ast}}$ we use intergration by parts to express
\begin{equation}
\begin{split}
\ap(E_\ast,E^0_\ast-\V^0)=&\sum_{i\in J}
\bigg[\sum_{\tau\in\mesh_i}\inner{R_\tau,E^0_\ast-\Ip  E^0_\ast}_\tau
+\sum_{\sigma\in\edges_i}\left\{\inner{J_{\sigma,1},E^0_\ast-\Ip E^0_\ast}_\sigma+\inner{J_{\sigma,2},E^0_\ast-\Ip E^0_\ast}_\sigma\right\}\\
&+\sum_{\sigma\in\bedges_i}\left(\int_\sigma\U{\textstyle\dfrac{}{\n_\sigma}}\left[\Pip\Lap(E^0_\ast-\Ip E^0_\ast)\right]-\int_\sigma{\textstyle\dfrac{\U}{\n_\sigma}}\Pip\Lap(E^0_\ast-\Ip E^0_\ast)\right)\bigg],\\
\end{split}
\end{equation}
\begin{equation}
\begin{split}
\sum_{\tau\in\mesh_i}&\inner{R_\tau,E^0_\ast-\Ip E^0_\ast}_\tau
+\sum_{\sigma\in\edges_i}\left\{\inner{J_{\sigma,1},E^0_\ast-\Ip E^0_\ast}_\sigma+\inner{J_{\sigma,2},E^0_\ast-\Ip E^0_\ast}_\sigma\right\}\\
&\leq\cproj\bigg(\sum_{\tau\in\mesh_i}\eta^2_\mesh(\U,\tau)\bigg)^{1/2}\bigg(\sum_{\tau\in\mesh_i}\norm{E^0_\ast}{H^2(\omega_\tau)}^2\bigg)^{1/2}
\leq\cproj \cshape\eta_\mesh(\U,\Omega_i)\norm{E_\ast^0}{H^2(\Omega_i)}
\end{split}
\end{equation}
The boundary intergal terms will be control by the inconsistnt part of the spline solution 
\begin{equation}
\begin{split}
\sum_{\sigma\in\bedges_i}&\left(\int_\sigma\U{\textstyle\dfrac{}{\n_\sigma}}\left[\Pip\Lap(E^0_\ast-\Ip E^0_\ast)\right]-\int_\sigma{\textstyle\dfrac{\U}{\n_\sigma}}\Pip\Lap(E^0_\ast-\Ip E^0_\ast)\right)\leq \semi{\U^\perp}{P_i}\norm{E_\ast^0}{H^2(\Omega_i)}
\end{split}
\end{equation}
Together we arrive at an estimate for the first term in \eqref{lem:DiscreteEstimatorReliability}
\begin{equation}
\ap(E_\ast,E^0_\ast-\V^0)\leq\cproj \cshape\left(\eta_\mesh(\U,\Omega_\ast)+C_1\semi{\U^\perp}{P}\right)\norm{E_\ast^0}{H^2(\Omega_\ast)}
\end{equation}
To control the inconsistent term from \eqref{lem:DiscreteEstimatorReliability}, we follow the same reasoning made in \eqref{eq2:lem:EstimatorReliability} from Lemma \ref{lem:EstimatorReliability} to get
\begin{equation}
\begin{split}
\ap(E_\ast,\U^\perp)&\leq\frac{\ap(E_\ast,E_\ast)}{2}+\frac{C_2}{2}\semi{\U^\perp}{P}^2,
\end{split}
\end{equation}
where $C_2$ retains the same meaning as before.
 Noting that $E_\ast^0=E_\ast+\U^\perp$, $\norm{E_\ast^0}{H^2(\Omega_\ast)}\leq\norm{E_\ast}{H^2(\Omega_\ast)}+\norm{\U^\perp}{H^2(\Omega_\ast)}$
Invoking norm equivalence \eqref{eq:lem:PerpNormEquiv} Summing up we arrive 
\begin{equation}
\ap(E_\ast,E_\ast)\leq C_\rm{dRel,1}\eta^2_\mesh(\U,\Omega_\ast)+C_\rm{dRel,2}\semi{\U^\perp}{P}^2
\end{equation}
\end{proof}
The presence of negative powers in $\semi{\U^\perp}{\mesh}$ on the right-hand side in \eqref{eq:res:lem:EstimatorReliability} and \eqref{eq:result:lem:DiscreteEstimatorReliability} may appear to pose a problem with decreasing mesh-size along the boundary. With the following realization from \cite{al2018adaptivity} we have shown that contributions from domain boundary integrals are dominated by the those coming from the mesh interior. 
\begin{lemma}\label{lem:boundarycontrol}
For sufficiently large stabilization terms $\gamma_1$ and $\gamma_2$,
\begin{equation}\label{eq:lem:boundarycontrol}
(\gamma_1-C_\rm{R})\norm{\U}{3/2,\mesh}^2+(\gamma_2-C_\rm{R})\left\norm{\textstyle\dfrac{\U}{\nu}\right}{1/2,\mesh}^2
\leq\Ccoer^{-1}\eta_\mesh^2(\U,\Omega)
\end{equation}
with $C_\rm{R}\lapprox\frac{\cshape}{\Ccoer}$.
\end{lemma}
\begin{remark}
From now on we let 
\begin{equation}\label{eq:newgamma}
\gamma:=\min\{\gamma_1-C_R,\gamma_2-C_R\}
\end{equation}
\end{remark}
\begin{corollary}\label{cor:aposteriori}
Under the assumptions of lemma \ref{lem:EstimatorReliability} and lemma \ref{lem:DiscreteEstimatorReliability}, if $\gamma>0$ then
\begin{equation}
\ap(\u-\U,\u-\U)\leq\Crel\etap^2(\U,\Omega),
\end{equation}
and
\begin{equation}
\enorm{\U^0_\ast-\U}{\mesh_\ast}^2\leq\Cdrel\eta_\mesh^2(\U,\omega_{R_{\mesh\to\mesh_\ast}})+\gamma^{-1}\Ccoer^{-1}\eta^2_\mesh(\U,\Omega).
\end{equation}
\end{corollary}
\section{Convergence} 		
In section we show that the derived computable estimator \eqref{eq:estimator} when used to direct refinement will result in decreased error.
This will hinge on the estimator Lipschitz property of Lemma \ref{lem:elp}.
To show that procedure \eqref{eq:afem} exhibits convergence we must be able to relate the errors of consecutive discrete solutions.
In the conforming discrete method \eqref{eq:cdp} the symmetry of the bilinear form, consistency of the formulation and finite-element spline space nesting will readily provide that via Galerkin Pythagoras.
This is not the case in Nitsche's formulation \eqref{eq:ndp} since our formulation is no longer consistent with \eqref{eq:cwp}.
We recall some of the results needed for convergence. 
\begin{lemma}[Estimator Lipschitz property]\label{lem:elp}
Let $\mesh$ be a partition of $\Omega$ satisfying conditions \eqref{eq:sr}.
There exists a constant $\Clip>0$, depending only $\cshape$, such that for any cell $\face\in\mesh$ we have
\begin{equation}\label{eq:result:lem:elp}
|\eta_\mesh (\V,\face)-\eta_\mesh (\W,\face)|
\leq\Clip\semi{\V-\W}{H^2(\omega_\face)},
\end{equation}
holding for every pair of finite-element splines $\V$ and $\W$ in $\X_\mesh$.
\end{lemma}


\begin{lemma}[Estimator error reduction]\label{lem:EstimatorReliabilityr}
Let $\mesh$ be a partition of $\Omega$ satisfying conditions \eqref{eq:sr},
let $\marked\subseteq\mesh$
and let $\mesh_\ast=\mathbf{REFINE}\,[\mesh,\marked]$.
There exists constants $\lambda\in(0,1)$ and $\Cest>0$, depending only on $\cshape$,
such that for any $\delta>0$ it holds that for any pair of finite-element splines $\V\in\X_\mesh$ and $\V_\ast\in\X_{\mesh_\ast}$
we have
\begin{equation}\label{eq:result:lem:EstimatorReliabilityr}
\eta_{\mesh_\ast}^2(\V_\ast,\Omega)
\leq(1+\delta)\left\{\eta_\mesh^2(\V,\Omega)-{\textstyle\frac{1}{2}}\eta_\mesh^2(\V,\marked)\right\}+\cshape(1+{\textstyle\frac{1}{\delta}})\enorm{\V-\V_\ast}{\mesh_\ast}^2.
\end{equation}
\end{lemma}
In what follows we establish estimates that allows us to compare two spline solutions on different admissible meshes. This replaces the unavailable Galerkin Pythagorus which the confomrning formulation enjoyed.
\begin{lemma}[Mesh perturbation]\label{lem:MP}
Let $\mesh$ and $\mesh_\ast$ be successive partitions satisfying conditions \eqref{eq:sr} which are obtained by $\ms{REFINE}$.
Then for a constant $\Ccomp>0$, depending only on $\cshape$,
we have for any $\delta>0$
\begin{equation}
\app(\v,\v)\leq(1+4\delta\Ccoer)\ap(\v,\v)+\frac{\Ccomp}{\delta}\left(\gamma_1\norm{\v}{3/2,\mesh}^2+\gamma_2\left\norm{\textstyle\dfrac{\v}{\nu}\right}{1/2,\mesh}^2\right),
\end{equation}
holding for every function $\v\in H^2(\Omega)$.
\end{lemma}
\begin{proof}
%
Given any $\v\in H^2(\Omega)$ we write
\begin{equation}\label{eq1:lem:MP}
\begin{split}
\app(\v,\v)=&\ap(\v,\v)+2\bigg(\int_\Gamma\Pip(\Lap\v){\textstyle\dfrac{\v}{\nu}}
-\int_\Gamma{\textstyle\dfrac{\Pip(\Lap\v)}{\nu}}\v\bigg)
-\gamma_1\left(\norm{\v}{\mesh,3/2}^2-\norm{\v}{\mesh_\ast,3/2}^2\right)\\
&-2\bigg(\int_\Gamma\Pipp(\Lap\v){\textstyle\dfrac{\v}{\nu}}
-\int_\Gamma{\textstyle\dfrac{\Pipp(\Lap\v)}{\nu}}\v\bigg)
-\gamma_2\left(\left\norm{\textstyle\dfrac{\v}{\nu}\right}{\mesh,1/2}^2
-\left\norm{\textstyle\dfrac{\v}{\nu}\right}{\mesh_\ast,1/2}^2\right).
\end{split}
\end{equation}
%
Look at the boundary integral terms depending on $\mesh$. Let $\sigma\in\bedges_\mesh$ an edge to some cell $\tau\in\mesh$, 
\begin{equation}\label{eq2:lem:MP}
\begin{split}
\int_\sigma\Pip(\Lap\v)\dfrac{\v}{\n_\sigma}&\leq\norm{\Pip(\Lap\v)}{\sigma}\left\norm{\textstyle\dfrac{\v}{\n_\sigma}\right}{\sigma}
\leq\cdtrace\cproj h_\sigma^{-1/2}\norm{\Lap\v}{\tau}\left\norm{\textstyle\dfrac{\v}{\n_\sigma}\right}{\sigma}.
\end{split}
\end{equation}
%
Summing \eqref{eq2:lem:MP} over all $\sigma\in\bedges_\mesh$ and an application of Schwarz's inequality on the summation would give
\begin{equation}
\begin{split}
\bigg|\int_\Gamma\Pip(\Lap\v){\textstyle\dfrac{\v}{\nu}}\bigg|&
\lapprox\bigg(\sum_{\sigma\in\bedges_\mesh}h_\sigma^{-1}\left\norm{\textstyle\dfrac{\v}{\n_\sigma}\right}{\sigma}^2\bigg)^{1/2}
\bigg(\sum_{\tau\in\mesh:\diff\tau\cap\Gamma\neq\emptyset}\norm{\Lap\v}{\tau}^2\bigg)^{1/2}\\
&\leq\left\norm{\textstyle\dfrac{\v}{\nu}\right}{\mesh,1/2}\norm{\Lap\v}{L^2(\Omega)}.
\end{split}
\end{equation}
Similarly, using the inverse-estimate $\norm{\dfrac{\Pip(\Lap\v)}{\n_\sigma}}{\sigma}\leq\cinv h_\sigma^{-1}\norm{\Pip(\Lap\v)}{\sigma}$, we obtain
\begin{equation}
\bigg|\int_\Gamma{\textstyle\dfrac{\Pip(\Lap\v)}{\nu}}\v\bigg|\leq\cdtrace\cinv\cproj\norm{\v}{\mesh,3/2}\norm{\Lap\v}{L^2(\Omega)}.
\end{equation}
%
We carry the same reasoning for the remaining boundary integral.
Employing Young's inequality with $\delta>0$ we arrive at
\begin{equation}\label{eq3:lem:MP}
\begin{split}
\app(\v,\v)&\lapprox\ap(\v,\v)
+4\delta\norm{\Lap\v}{L^2(\Omega)}^2
+\left(\textstyle\frac{1}{\delta}+\gamma_1\right)\norm{\v}{\mesh,3/2}^2
+\left(\textstyle\frac{1}{\delta}+\gamma_1\right)\norm{\v}{\mesh_\ast,3/2}^2\\
&+\left(\textstyle\frac{1}{\delta}+\gamma_2\right)\left\norm{\textstyle\dfrac{\v}{\nu}\right}{\mesh,1/2}^2
+\left(\textstyle\frac{1}{\delta}+\gamma_2\right)\left\norm{\textstyle\dfrac{\v}{\nu}\right}{\mesh_\ast,1/2}^2.
\end{split}
\end{equation}
%
With the fact that $h_\sigma\leq\cshape h_{\sigma_\ast}$, with $\sigma\in\bedges_\mesh$ and $\sigma_\ast\in\bedges_{\mesh_\ast}$, we infer that $\norm{\v}{3/2,\mesh_\ast}\leq\cshape^{-1}\norm{\v}{\mesh,3/2}$ and $\norm{\dfrac{\v}{\nu}}{1/2,\mesh_\ast}\leq\cshape^{-1}\norm{\dfrac{\v}{\nu}}{1/2,\mesh}$.
\begin{equation}
\left(\textstyle\frac{1}{\delta}+\gamma_1\right)\left(\norm{\v}{\mesh_\ast,3/2}^2
+\norm{\v}{\mesh_\ast,3/2}^2\right)\leq\frac{\Ccomp\gamma_1}{\delta}\norm{\v}{\mesh,3/2}^2,
\end{equation}
where $\Ccomp>0$ is an appropriate proportionality parameter that depends on $\cshape$.
A similar argument holds for terms including boundary norms of $\dfrac{\v}{\nu}$.
\end{proof}
\begin{lemma}[Comparison of solutions]\label{lem:cos}
Let $\mesh$ and $\mesh_\ast$ be successive admissible partitions obtained by $\ms{REFINE}$
and let $\U\in\Xp$ and $\U_\ast\in\X_{\mesh_\ast}$ be the finite-element spline solutions to \eqref{eq:ndp}.
Then we have for any $\eps>0$
\begin{equation}\label{eq:result:lem:cos}
\begin{split}
\app(\epp,\epp)\leq&(1+\eps)\ap(\ep,\ep)
-\frac{\Ccoer}{2}\enorm{\upp-\up}{\mesh_\ast}^2
+\frac{C_\rm{Comp}}{\eps\gamma}\eta_\mesh^2
\end{split}
\end{equation}
\end{lemma}
\begin{proof}
We follow the following abbreviation. Let $e=\u-\U$, let $e_\ast=\u-\U_\ast$, let $E_\ast^0=\U^0_\ast-\U^0$, and let $E_\ast^\perp=\U^\perp_\ast-\U^\perp$. Partial Galerkin implies
\begin{equation}
\begin{split}
\app(\e_{\ast},\e_\ast)&=\app(\e_\ast,\e_\ast+E^0_\ast)
=\app(\e_\ast+E^0_\ast,\e_\ast+E^0_\ast)-\app(E_\ast^0,\e_\ast+E_\ast^0)
\end{split}
\end{equation}
and Partial Galerkin and symmetry again we have
\begin{equation}
\app(\e_{\ast},\e_\ast)=\app(\e_\ast+E^0_\ast,\e_\ast+E^0_\ast)-\app(E_\ast^0,E_\ast^0)
\end{equation}
Rewriting $\U_\ast-E_\ast^0=\U-E_\ast^\perp$ we can express $e_\ast+E_\ast^0=e-E_\ast^\perp$ and therefore
\begin{equation}
\app(e_\ast+E_\ast^0,e_\ast+E_\ast^0)
=\app(e,e)-2\app(e,E\ast^\perp)+\app(E_\ast^\perp,E_\ast^\perp)
\end{equation}
We then have
\begin{equation}
\begin{split}
\app(\e_{\ast},\e_\ast)&=\app(e,e)-2\app(e,E\ast^\perp)+\app(E_\ast^\perp,E_\ast^\perp)
-\app(E_\ast^0,E_\ast^0)
\end{split}
\end{equation}
Employ Young's inequality
\begin{equation}
\app(e,e)-2\app(e,E_\ast^\perp)\leq(1+\delta)\app(e,e)+\frac{\Ccont^2}{\delta\Ccoer}\enorm{E_\ast^\perp}{P_\ast}^2
\end{equation}
Writing $E_\ast^0=E_\ast-E_\ast^\perp$ and with $\enorm{E_\ast}{P_\ast}^2\leq2\enorm{E^0_\ast}{P_\ast}^2+2\enorm{E_\ast}{P_\ast}^2$ makes $\enorm{E_\ast^0}{\mesh_\ast}^2\ge\frac{1}{2}\enorm{E_\ast}{\mesh_\ast}^2-\enorm{E_\ast^\perp}{\mesh_\ast}^2$ and
\begin{equation}
\begin{split}
\app(E_\ast^\perp,E_\ast^\perp)-\app(E_0^\ast,E_\ast^0)&\leq\Ccont\enorm{E_\ast^\perp}{P_{\ast}}^2-\Ccoer\enorm{E_\ast^0}{\mesh_\ast}^2\\
&\leq-\frac{\Ccoer}{2}\enorm{E_\ast}{\mesh_\ast}^2+C_4\enorm{E_\ast^\perp}{P_\ast}^2
\end{split}
\end{equation}
where $C_4=\Ccoer+\Ccont$
We therefor have, with $C_5=\max\{C_4,\frac{\Ccont^2}{\Ccoer}\}$
\begin{equation}
\begin{split}
\app(\e_{\ast},\e_\ast)&\leq
(1+\delta)\app(e,e)-\frac{\Ccoer}{2}\enorm{E_\ast}{\mesh_\ast}^2+C_5\left(1+\frac{1}{\delta}\right)\enorm{E_\ast^\perp}{P_\ast}^2
\end{split}
\end{equation}
Using the fact that edge sizes between two consequetive refinement steps are comparable and \eqref{eq:lem:PerpNormEquiv}
\[
\enorm{E_\ast^\perp}{\mesh_\ast}^2\lapprox\semi{\U^\perp_\ast}{\mesh_\ast}^2+\semi{\U^\perp}{\mesh}^2
\lapprox\frac{\Ccoer^{-1}}{\gamma}\left(\eta^2_{\mesh_\ast}(\Omega)+\eta^2_{\mesh}(\Omega)\right)
\]
In view of Lemma \ref{lem:MP}, for the same $\delta>0$ above, and Lemma \eqref{eq:lem:boundarycontrol}
\begin{equation}
\begin{split}
\app(e,e)\leq&(1+4\delta\Ccoer)\ap(e,e)+\frac{\Ccomp\Ccoer^{-1}}{\delta\gamma}\eta^2_\mesh(\Omega).
\end{split}
\end{equation}
%
Summing up
\begin{equation}
\begin{split}
\app(\e_\ast,\e_\ast)&\leq(1+C\delta)\ap(e,e)-\frac{\Ccoer}{2}\enorm{E_\ast}{\mesh_\ast}^2+\frac{C_\rm{Comp}}{\delta\gamma}\left(\eta^2_{\mesh_\ast}(\Omega)+\eta^2_{\mesh}(\Omega)\right),
\end{split}
\end{equation}
where $C$ and $C_\rm{Comp}$ depend on $\Ccoer$ and $\Ccont$.

\end{proof}

\begin{theorem}[Convergence of Nitsche's AFEM]\label{thm:nc}
Given $f\in L^2(\Omega)$ and Dolfer parameter $\theta\in(0,1]$, there exists $\gamma_C(\theta)>0$, a contractive factor $\alpha\in(0,1)$ and a constant $\Cest>0$,  such that for all $\gamma\ge\gamma_C$ the adaptive procedure $\mathbf{AFEM}\,[\mesh,f,\theta]$ with produce two successive solutions $\U\in\Xp$ and $\U_\ast\in\Xpp$ to problem \eqref{eq:ndp} for which
\begin{equation}\label{eq:result:thm:ConvNit}
\app+\Cest\etapp^2\leq\alpha\left(\ap
+\Cest\etap^2\right).
\end{equation}
\end{theorem}
\begin{proof}
Adopt the following abbreviations:
\begin{eqnarray}
\a_\mesh=\a_\mesh(\u-\U,\u-\U),&E_\ast=\enorm{\U-\U_\ast}{\mesh_\ast},\\
\eta_\mesh=\eta_{\mesh}(\U,\mesh),&\eta_\mesh(\marked)=\eta_{\mesh}(\U,\marked).
\end{eqnarray}
Let $\Cest^{-1}=\cshape(1+\frac{1}{\delta})\frac{2}{\Ccoer}$.
In view of Lemma \ref{lem:cos},
\begin{equation}\label{eq1:thm:nc}
\app+\Cest\etapp^2\leq(1+\eps)\ap-{\textstyle\frac{\Ccoer}{2}}E_\ast^2
+\textstyle{\frac{\Ccomp}{\eps\gamma}}\etap^2+\Cest\etapp^2.
\end{equation}
By invoking Lemma \ref{lem:EstimatorReliabilityr} on $\Cest\etapp^2$
\begin{equation}
\begin{split}
\app+\Cest\etapp^2&\leq(1+\eps)\ap
+\textstyle{\frac{\Ccomp}{\eps\gamma}}\etap^2-{\textstyle\frac{\Ccoer}{2}}E_\ast^2\\
&\quad+\Cest\left[(1+\delta)\left\{\etap^2-{\textstyle\frac{1}{2}}\etap^2(\marked)\right\}+\cshape(\textstyle1+\frac{1}{\delta})E_\ast^2\right],
\end{split}
\end{equation}
eliminates $E_\ast$ from the previous expression.
From Dorler $-\etap^2(M)\leq\theta^2\etap^2$
and in view of Corollary \ref{cor:aposteriori},
\begin{equation}
\begin{split}
\Cest(1+\delta)\left\{\etap^2-{\textstyle\frac{1}{2}}\etap^2(\marked)\right\}
&\leq\Cest(1+\delta)\etap^2-\Cest(1+\delta){\textstyle\frac{\theta^2}{2}}\etap^2\\
&\leq\Cest(1+\delta)\etap^2-\Cest(1+\delta){\textstyle\frac{\theta^2}{2}}\left(\textstyle\frac{1}{2}\etap^2+\frac{1}{2\Crel}\ap\right).
\end{split}
\end{equation}
Expression \eqref{eq1:thm:nc} now reads
\begin{equation}
\app+\Cest\etapp^2\leq\left(1+\eps-\Cest(1+\delta){\textstyle\frac{\theta^2}{4\Crel}}\right)\ap
+\left(\textstyle{\frac{\Ccomp}{\eps\gamma}}+\Cest(1+\delta)\left(1-\frac{\theta^2}{4}\right)\right)\etap^2.
\end{equation}
Noting that $\Cest(1+\delta)=\delta\frac{\Ccoer}{2\cshape}$ we arrive at
\begin{equation}
\app+\Cest\etapp^2\leq\left(1+\eps-{\textstyle\frac{\delta\theta^2\Ccoer}{8\cshape\Crel}}\right)\ap
+\Cest\left(\textstyle{\frac{\Ccomp}{\eps\gamma\Cest}}+(1+\delta)\left(1-\frac{\theta^2}{4}\right)\right)\etap^2.
\end{equation}
It what remains we verify the existence of $\eps>0,\delta>0$ and $\gamma_C(\theta)>0$ such that for all $\gamma>\gamma_C$ the factors $1+\eps-{\textstyle\frac{\delta\theta^2\Ccoer}{8\cshape\Crel}}$ and $\textstyle{\frac{\Ccomp}{\eps\gamma\Cest}}+(1+\delta)\left(1-\frac{\theta^2}{4}\right)$ are positive and less that $1$.
Let $\Lambda_1=\frac{\Ccoer}{8\cshape\Crel}$ and $\Lambda_2=\frac{2\Ccomp\cshape}{\Ccoer}$.
Then the corresponding conditions will read
\begin{equation}\label{eq1:lem:nc}
\textstyle
0<1+\eps-\delta\theta^2\Lambda_1<1
\quad\text{and}\quad
0<(1+{\textstyle\frac{1}{\delta}})\frac{\Lambda_2}{\eps\gamma}+(1+\delta)\left(\textstyle1-\frac{\theta^2}{4}\right)<1.
\end{equation}
For any $\delta>0$ let $\eps=\frac{\delta\theta^2}{2}\Lambda_1$ so that the first condition in \eqref{eq1:lem:nc} holds and let $\delta=\frac{\theta^2}{2-\theta^2}$ so that $(1+\delta)\left(\textstyle1-\frac{\theta^2}{4}\right)=1-\frac{\theta^2}{2}$ then pick $\gamma$ sufficiently large so that $(1+{\textstyle\frac{1}{\delta}})\frac{\Lambda_2}{\eps\gamma}<\frac{\theta^2}{2}$ to obtain the second relation in \eqref{eq1:lem:nc}. We note that the $\gamma_C(\theta):=\frac{2(1+\frac{1}{\delta}\Lambda_2)}{\theta^2\eps}$.
\end{proof}
\begin{remark}\label{rem:ConvergenceAlpha}
We may define contractive factor $\alpha(\delta):=\max\left\{\frac{1}{2},(1+{\textstyle\frac{1}{\delta}})\frac{\Lambda_2}{\eps\gamma}+1-\frac{\theta^2}{2}\right\}$ with the specified $\delta$ above. In combination with the $\gamma>\gamma_C(\theta)$ we also have $(1+{\textstyle\frac{1}{\delta}})\frac{\Lambda_2}{\eps\gamma}+1-\frac{\theta^2}{2}<1-c\theta^2$ for some $c$.
\end{remark}
\section{Quasi-optimlaity of AFEM}
The total-error norm is given by
\begin{equation}
\rho_\mesh(\v,\V,g)=\left(\enorm{\v-\V}{\mesh}^2+\osc_\mesh^2(g)\right)^{1/2}.
\end{equation}
The AFEM approximation class defined by the total-error norm is then given by
\begin{equation}
\bb{A}^s=\left\{\v\in H_0^2(\Omega):\sup_{N>0}N^sE_\mesh(\v)<\infty\right\},
\end{equation}
where
\begin{equation}
E_\mesh(\v)=\inf_{\V\in\Xp}\rho_\mesh(\v,\V,\cal{L}v)
,\quad\v\in H^2(\Omega).
\end{equation}
Analogously, we define the approximation class in which approximation comes from boundary conforming spline spaces by
\begin{equation}
\bb{A}^s_0=\left\{\v\in H_0^2(\Omega):\sup_{N>0}N^sE^0_\mesh(\v)<\infty\right\}
\end{equation}
where 
\begin{equation}
E_\mesh^0(\v)=\inf_{\V_0\in\Xp^0}\left(\norm{\v-\V_0}{H^2(\Omega)}^2+\osc_\mesh^2(\cal{L}\v)\right)^{1/2}
,\quad\v\in H^2_0(\Omega)
\end{equation}

\begin{lemma}[Equivalence of classes]\label{lem:ClassEquivalence}
$\bb{A}^s=\bb{A}^s_0$
\end{lemma}
\begin{proof}
Let $u\in\bb{A}_s$, for $s>0$, let $N>\#\mesh_0$, let ${P}_\ast\in\scr{P}_N$ and let $\V_\ast\in\X_\ast$ be such that
\begin{equation}
\rho_{\mesh_\ast}(\u,\V_\ast,f)=\inf_{\mesh\in\scr{P}_N}E_\mesh(\u)
\end{equation}
Using the triangle inquality $\enorm{\u-\V^0_\ast}{\mesh_\ast}\leq\enorm{\u-\V_\ast}{\mesh_\ast}+\enorm{\V_\ast-\V^0_\ast}{\mesh_\ast}$ with the fact that $\semi{\V_\ast}{\mesh_\ast}=\semi{u-\V_\ast}{\mesh_\ast}$ we have in view of norm equivalence \eqref{eq:lem:PerpNormEquiv}
\begin{equation}
\enorm{\V_\ast-\V^0_\ast}{\mesh_\ast}\leq\Cperp\semi{\V_\ast}{\mesh_\ast}\lapprox\enorm{u-\V_\ast}{\mesh_\ast},
\end{equation}
from which we obtain
\begin{equation}
\enorm{\u-\V^0_\ast}{\mesh_\ast}^2+\osc^2_{\mesh_\ast}(f)\lapprox\enorm{u-\V_\ast}{\mesh_\ast}^2+\osc^2_{\mesh_\ast}(f).
\end{equation}
Upon taking infimum we arrive at
\begin{equation}
\enorm{\u-\V^0_\ast}{\mesh_\ast}^2+\osc^2_{\mesh_\ast}(f)\lapprox E^2_P(\u,f)\lapprox N^{-2s}.
\end{equation}
\end{proof}

\begin{lemma}[Quasi-optimality of total error]\label{lem:QuasiOptimalityTotalError}
Let $\u$ be the solution of \eqref{eq:cwp} and for all $\mesh\in\scr{P}$ let $\U\in\Xp$ be the discrete solution to \eqref{eq:dnp}. Then, for a constant $C_\rm{QOTE}>0$ and $\gamma_Q>0$ we have for all $\gamma\ge\gamma_Q$
\begin{equation}
\rho^2_\mesh(\u,\U,f)\leq C_\rm{QOTE}\inf_{\V\in\Xp}\rho^2_\mesh(\u,\V,f).
\end{equation}
\end{lemma}
\begin{proof}
In view of Coercivity \eqref{eq:Coercivity}, partial Galerkin orthogonality \eqref{eq:PartialGalerkin} and Continuity \eqref{eq:Continuity}
\begin{equation}
\begin{split}
\Ccoer\enorm{\e}{P}^2&\leq\ap(\e,\u-\U)=\ap(\e,\u-\U^0)-\ap(\e,\U^\perp)\\
&=\ap(\e,\u-\V_0)+\ap(\e,\U^\perp)=\ap(\e,\u-\V)+\ap(\e,\V^\perp)+\ap(\e,\U^\perp)\\
&\leq\Ccont\enorm{\e}{\mesh}\left(\enorm{\u-\V}{\mesh}+\enorm{\V^\perp}{\mesh}+\enorm{\U^\perp}{\mesh}\right)
\end{split}
\end{equation}
Norm equivalence \eqref{lem:PerpNormEquiv} $\enorm{\V^\perp}{\mesh}\leq\Cperp\semi{\u-\V^\perp}{\mesh}\leq\enorm{\u-\V}{\mesh}$. Nonconforming control \eqref{eq:lem:boundarycontrol} and Global Lower Bound \eqref{eq:result:lem:EstimatorEfficiency}
makes $\enorm{\U^\perp}{\mesh}\lapprox\gamma^{-1/2}\eta_\mesh\leq\gamma^{-1/2}\Ceff\rho_\mesh(\u,\U,f)$.
From
\begin{equation}
\Ccoer\enorm{\e}{P}\lapprox\Ccont\left(\enorm{\u-\V}{\mesh}+\gamma^{-1/2}\Ceff\rho_\mesh(\u,\U,f)\right)
\end{equation}
we get
\begin{equation}
\enorm{\e}{\mesh}^2\lapprox\frac{\Ccont^2}{\Ccoer^2}\left(\enorm{\u-\V}{\mesh}^2+\gamma^{-1}\Ceff^2\rho^2_\mesh(\u,\U,f)\right)
\end{equation}
Add $\osc_\mesh^2(f)$ to the preceding expression to get
\begin{equation}
\left(1-\frac{\Ccont^2\Ceff^2}{\Ccoer^2}\gamma^{-1}\right)\rho_\mesh^2(\u,\U,f)\lapprox\frac{\Ccont^2\Ceff^2}{\Ccoer^2}\rho_\mesh^2(\u,\V,f).
\end{equation}
Let $\gamma_Q:=\frac{\Ccont^2\Ceff^2}{\Ccoer^2}$.
\end{proof}
Let 
\begin{equation}
\theta_\ast(\gamma):=\bigg(\frac{\Ceff-2\Cdrel\gamma^{-1}}{2(1+\Cdrel)}\bigg)^{1/2}
\quad\text{and}\quad\gamma_\ast(\theta):=\max\left(\frac{2\Cdrel}{\Ceff},\gamma_Q,\gamma_C(\theta)\right).
\end{equation}
Then $\theta_\ast>0$ and since $\Ceff<\Cdrel$, $\theta_\ast<1$.

\begin{lemma}[Optimal marking]\label{lem:OptimalMarking}
Let $\U=\ms{SOLVE}\,[\mesh,f]$, let $\mesh_\ast$ be any refinement of $\mesh$ and let $\U_\ast=\ms{SOLVE}\,[\mesh_\ast,\eff]$. If for some positive $\mu<1$
\begin{equation}\label{eq:lem:OptimalMarking:assumption}
\enorm{u-\U^0_\ast}{\mesh_\ast}^2+\rm{osc}_\ast^2(\eff,\mesh_\ast)\leq\mu\big(\enorm{u-\U}{}^2+\osc_\mesh^2(\eff,\mesh)\big),
\end{equation}
and $R_{P\to P_\ast}$ denotes collection of all elements in $P$ requiring refinement to obtain $P_\ast$ from $P$, then for $\theta\in(0,\theta_\ast(\gamma))$ we have 
\begin{equation}
\eta_\mesh (\U,\omega_{R_{P\to P_\ast}})\ge\theta\eta_\mesh (\U,\Omega)
\end{equation}
\end{lemma}
\begin{proof}
Let $\theta<\theta_\ast$, the parameter $\theta_\ast$ to be specified later, such that the linear contraction of the total error holds for 
\begin{equation}
\mu(\theta,\gamma):=\frac{1}{2}\left(1-\frac{2\Cdrel\gamma^{-1}}{\Ceff}\right)\left(1-\frac{\theta^2}{\theta^2_\ast}\right)<\frac{1}{2},\quad(\gamma\ge\gamma_\ast).
\end{equation}
The efficiency estimate \eqref{eq:result:lem:EstimatorEfficiency} together with the assumption \eqref{eq:lem:OptimalMarking:assumption}
\begin{equation}
\begin{split}
(1-2\mu)\Ceff\eta_\mesh^2(\U,\mesh)&\leq(1-\mu)\rho_\mesh^2(\u,\U,f)\\
&=\rho_\mesh^2(\u,\U,f)-\rho_\ast^2(\u_\ast,\U^0_\ast,f)\\
&=\enorm{\u-\U}{\mesh}^2-2\enorm{\u-\U^0_\ast}{\mesh_\ast}^2+\osc_P^2(f,\Omega)-2\osc_{\mesh_\ast}^2(f,\Omega)
\end{split}
\end{equation}
Triangle inequality and Discrete Reliability \eqref{eq:result:lem:DiscreteEstimatorReliability}
\begin{equation}
\begin{split}
\enorm{\u-\U}{\mesh}^2-2\enorm{\u-\U^0_\ast}{\mesh_\ast}^2&\leq2\enorm{\U^0_\ast-\U}{\mesh}^2\\
&\leq 2C_\rm{dRel}\left(\eta^2_\mesh(\U,\omega_{R_{\mesh\to\mesh_\ast}})
+\gamma^{-1}\eta_\mesh^2(\U,\Omega)\right)
\end{split}
\end{equation}
Estimator Dominance over oscillation
\begin{equation}
\begin{split}
\osc_P^2(f,\Omega)-2\osc_{\mesh_\ast}^2(f,\Omega)&\leq2\osc^2_\mesh(f,\omega_{R_{\mesh\to\mesh_\ast}})\leq2\eta_\mesh^2(\U,\omega_{R_{\mesh\to\mesh_\ast}})
\end{split}
\end{equation}
From
\begin{equation}
(1-2\mu)\Ceff\eta_\mesh^2(\U,\mesh)\leq2(1+\Cdrel)\eta_\mesh^2(\U,\omega_{R_{\mesh\to\mesh_\ast}})+2\Cdrel\gamma^{-1}\eta_\mesh^2(\U,\Omega)
\end{equation}
re-write into
\begin{equation}
\left((1-2\mu)\Ceff+2\Cdrel\gamma^{-1}\right)\eta_\mesh^2(\U,\mesh)\leq2(1+\Cdrel)\eta_\mesh^2(\U,\omega_{R_{\mesh\to\mesh_\ast}}).
\end{equation}
For reader clarity we show that
\[
\frac{(1-2\mu)\Ceff-2\Cdrel\gamma^{-1}}{2(1+\Cdrel)}=\theta^2.
\]
Express
\[
(1-2\mu)\Ceff-2\Cdrel\gamma^{-1}=\theta^22(1+\Cdrel)=\frac{\theta^2(\Ceff-2\Cdrel\gamma^{-1})}{\theta^2_\ast},
\]
which is same as
\[
\begin{split}
-2\mu=\frac{\theta^2}{\theta_\ast^2}\left(1-\frac{2\Cdrel\gamma^{-1}}{\Ceff}\right)+\frac{2\Cdrel\gamma^{-1}}{\Ceff}-1
=\left(1-\frac{2\Cdrel\gamma^{-1}}{\Ceff}\right)\left(\frac{\theta^2}{\theta^2_\ast}-1\right).
\end{split}
\]
\end{proof}
\begin{lemma}[Cardinality of Marked Cells]\label{lem:MarkedComplexity}
Let $\{(\mesh_\ell,\bb{X}_\ell,\U_\ell)\}_{\ell\ge0}$ be sequence generated by $\ms{AFEM}\,(\mesh_0,\eff;\eps,\theta)$ for admissible $P_0$ and the pair $\u\in\bb{A}^s$ for some $s>0$ then 
\begin{equation}\label{eq:lem:MarkedComplexity}
\#\scr{M}_\ell\lapprox\left(1-\frac{\theta^2}{\theta_\ast^2}\right)^{-\frac{1}{2s}}\semi{u}{\bb{A}_s}^{-\frac{1}{s}}\rho_\ell(\u,\U_\ell,f)^{-\frac{1}{s}}
\end{equation}
\end{lemma}
\begin{proof}
Let $(\u,f)\in\bb{A}_s$ and set $\eps^2=\mu C_\rm{QOTE}^{-1}\rho_\ell^2(\u,\U_\ell,f)$.
In view of Lemma \ref{lem:ClassEquivalence}, $\u\in\bb{A}_s^0$ and there exists an admissible partition $\mesh_\eps$ and $\V^0_\eps\in\X^0_\eps$ with $\rho_\eps^2(\u,\V^0_\eps,f)\leq\eps^2$ and $\#\mesh_\eps\lapprox\semi{\u}{\bb{A}^s}^{1/s}\eps^{-1/s}$.
Let $\mesh_\ast$ be the overlay of meshes $\mesh_\ell$ and $\mesh_\eps$.
From \eqref{eq:lem:SolveConformPart} 
\begin{equation}
\a_{\mesh_\ast}(\U^0_\ast,\W^0)=\ellf(\W^0)\quad\forall\W^0\in\X_\ast^0,
\end{equation}
we invoke Lemma \ref{lem:QuasiOptimalityTotalError} on $\U^0_\ast$ and use the fact $\mesh_\ast\ge\mesh_\eps$ makes $\X_\ast\supseteq\X_\eps$ and obtain
\begin{equation}
\rho^2_\ast(\u,\U^0_\ast,f)\leq C_\rm{QOTE}\rho^2_\eps(\u,\V^0_\eps,f)\leq\eps^2=\mu\rho^2_\ell(\u,\U_\ell,f)
\end{equation}
We may now invoke Lemma \ref{lem:OptimalMarking} and $R_{\mesh_\ell\to\mesh_\ast}$ satisfies Dorfler property 
Minimal cardinality of marked cells
\begin{equation}
\#\scr{M}_\ell\leq\#R_{\mesh_\ell\to\mesh_\ast}\leq\#\mesh_\ast-\#\mesh_\ell
\end{equation}
In view of mesh overlay property $\#\mesh_\ast\leq\mesh_\eps+\#\mesh_\ell-\#\mesh_0$ in \eqref{eq:MeshOverelay} and definition of $\eps$ we arrive at
\begin{equation}
\#\scr{M}_\ell\leq\#\mesh_\eps-\#\mesh_0\lapprox\mu^{-1/2s}\semi{u}{\bb{A}^s}^{1/s}\rho_\ell(\u,\U_\ell,f)^{-1/s}
\end{equation}
\end{proof}
\begin{theorem}[Quasi-optimality]
Let $\gamma_\ast$ and $\theta_\ast$ be as above. If $\gamma>\gamma_\ast$ and $\theta\in(0,\theta_\ast(\gamma))$, $\u\in\bb{A}^s$ and $P_0$ is admissible, then the call $\mathbf{AFEM}\,[\mesh_0,\eff,\eps,\theta]$ generates a sequence $\{(\mesh_\ell,\bb{X}_\ell,\U_\ell)\}_{\ell\ge0}$ of strictly admissible partitions $\mesh_\ell$, conforming finite-element spline spaces $\bb{X}_\ell$ and discrete solutions $\U_\ell$ satisfying
\begin{equation}
\rho_\ell(\u,\U_\ell,\eff)
\lapprox\Phi(s,\theta)\semi{(\u,\eff)}{\bb{A}_s}^{}(\#\mesh-\#\mesh_0)^{-s}
\end{equation}
with $\Phi(s,\theta)=(1-{\theta^2}/{\theta_\ast^2})^{-\frac{1}{2}}$
\end{theorem}
\begin{proof}
The proof is similar to that of the confomring forumlation \cite{}. For completeness we outline the analysis.
Let $\theta<\theta_\ast$ be given and assume that $u\in\bb{A}^s(\rho)$. 
We will show that the adaptive procedure $\ms{AFEM}$ will produce a sequence $\{(\mesh_\ell,\X_\ell,\U_\ell)\}_{\ell\ge0}$ such that $\rho_\ell\lapprox(\#\mesh_\ell-\#\mesh_0)^{-s}$.
In view of Convergence Theorem \ref{thm:nc}, we have for a factor $\Cest>0$ and a contractive factor $\alpha\in(0,1)$, Efficiency Estimate \eqref{eq:result:lem:EstimatorEfficiency} and Estimator Dominance \eqref{eq:EstimatorDominance}
\begin{equation}
\sum_{j=0}^{\ell-1}\rho_j^{-\frac{1}{s}}\leq\sum_{j=0}^{\ell-1}\alpha^{\frac{\ell-j}{s}}\textstyle\left(1+\frac{\Cest}{\Ceff}\right)^{\frac{1}{2s}}\left(e_\ell^2+\Cest\osc_\ell^2\right)^{-\frac{1}{2s}}.
\end{equation}
Cardinality of Marked Cells \eqref{eq:lem:MarkedComplexity} and \eqref{eq:MarkingComplexity} yields
\begin{equation}
\#\mesh_\ell-\#\mesh_0\lapprox
\semi{u}{\bb{A}^s}^{-{1}/{s}}
\left(1+\frac{\Cest}{\Ceff}\right)^{{1}/{2s}}
\frac{\alpha^{1/s}}{1-\alpha^{1/s}}
\left(1-\frac{\theta^2}{\theta_\ast^2}\right)^{-{1}/{2s}}\rho_\ell(\u,\U_\ell,f)^{-\frac{1}{s}}
\end{equation}
From Remark \ref{rem:ConvergenceAlpha}
\begin{equation}
\frac{\alpha^{1/s}}{1-\alpha^{1/s}}\leq
\end{equation}
\end{proof}
\section{Acknowledgements}
We thank Emmanuil Georgoulis for discussion about dG methods and his invaluable advice.
\bibliography{publications}

\begin{thebibliography}{10}

\bibitem{adams1975sobolev}
{\sc R.~A. Adams}, {\em Sobolev spaces. 1975}, Academic Press, New York, 1975.

\bibitem{ainsworth2011posteriori}
{\sc M.~Ainsworth and J.~T. Oden}, {\em A posteriori error estimation in finite
  element analysis}, vol.~37, John Wiley \& Sons, 2011.

\bibitem{al2018adaptivity}
{\sc I.~Al~Balushi, W.~Jiang, G.~Tsogtgerel, and T.-Y. Kim}, {\em Adaptivity of
  a b-spline based finite-element method for modeling wind-driven ocean
  circulation}, Computer Methods in Applied Mechanics and Engineering, 332
  (2018), pp.~1--24.

\bibitem{babuvska1978posteriori}
{\sc I.~Babu{\v{s}}ka and W.~C. Rheinboldt}, {\em A-posteriori error estimates
  for the finite element method}, International Journal for Numerical Methods
  in Engineering, 12 (1978), pp.~1597--1615.

\bibitem{babuvvska1978error}
{\sc I.~Babuv{\v{s}}ka and W.~C. Rheinboldt}, {\em Error estimates for adaptive
  finite element computations}, SIAM Journal on Numerical Analysis, 15 (1978),
  pp.~736--754.

\bibitem{bazilevs2006isogeometric}
{\sc Y.~Bazilevs, L.~Beirao~da Veiga, J.~A. Cottrell, T.~J. Hughes, and
  G.~Sangalli}, {\em Isogeometric analysis: approximation, stability and error
  estimates for h-refined meshes}, Mathematical Models and Methods in Applied
  Sciences, 16 (2006), pp.~1031--1090.

\bibitem{bazilevs2007weak}
{\sc Y.~Bazilevs and T.~J. Hughes}, {\em Weak imposition of dirichlet boundary
  conditions in fluid mechanics}, Computers \& Fluids, 36 (2007), pp.~12--26.

\bibitem{binev2004adaptive}
{\sc P.~Binev, W.~Dahmen, and R.~DeVore}, {\em Adaptive finite element methods
  with convergence rates}, Numerische Mathematik, 97 (2004), pp.~219--268.

\bibitem{bonito2010quasi}
{\sc A.~Bonito and R.~H. Nochetto}, {\em Quasi-optimal convergence rate of an
  adaptive discontinuous galerkin method}, SIAM Journal on Numerical Analysis,
  48 (2010), pp.~734--771.

\bibitem{bubuvska1984feedback}
{\sc I.~Bubu{\v{s}}ka and M.~Vogelius}, {\em Feedback and adaptive finite
  element solution of one-dimensional boundary value problems}, Numerische
  Mathematik, 44 (1984), pp.~75--102.

\bibitem{cascon2008quasi}
{\sc J.~M. Cascon, C.~Kreuzer, R.~H. Nochetto, and K.~G. Siebert}, {\em
  Quasi-optimal convergence rate for an adaptive finite element method}, SIAM
  Journal on Numerical Analysis, 46 (2008), pp.~2524--2550.

\bibitem{dorfler1996convergent}
{\sc W.~D{\"o}rfler}, {\em A convergent adaptive algorithm for poisson's
  equation}, SIAM Journal on Numerical Analysis, 33 (1996), pp.~1106--1124.

\bibitem{embar2010imposing}
{\sc A.~Embar, J.~Dolbow, and I.~Harari}, {\em Imposing dirichlet boundary
  conditions with nitsche's method and spline-based finite elements},
  International journal for numerical methods in engineering, 83 (2010),
  pp.~877--898.

\bibitem{buffa2016adaptive}
{\sc A. Buffa and C. Giannelli},
{\em Adaptive isogeometric methods with hierarchical splines: error estimator and convergence}, {Mathematical Models and Methods in Applied Sciences}, {volume 26}, {number 01}, {pages 1--25}, {2016}, {World Scientific}

\bibitem{buffa2016complexity}
{\sc A. Buffa, C. Giannelli, P. Morgenstern and D. Peterseim}
{\em Complexity of hierarchical refinement for a class of admissible mesh configurations}, {Computer Aided Geometric Design}, {volume 47}, {pages 83--92}, {year 2016}, {Elsevier}.


\bibitem{feischl2014adaptive}
{\sc M.~Feischl, T.~F�hrer, and D.~Praetorius}, {\em Adaptive fem with optimal
  convergence rates for a certain class of nonsymmetric and possibly nonlinear
  problems}, SIAM Journal on Numerical Analysis, 52 (2014), pp.~601--625.

\bibitem{grisvard2011elliptic}
{\sc P.~Grisvard}, {\em Elliptic problems in nonsmooth domains}, vol.~69, SIAM,
  2011.

\bibitem{hughes2005isogeometric}
{\sc T.~J. Hughes, J.~A. Cottrell, and Y.~Bazilevs}, {\em Isogeometric
  analysis: Cad, finite elements, nurbs, exact geometry and mesh refinement},
  Computer methods in applied mechanics and engineering, 194 (2005),
  pp.~4135--4195.

\bibitem{juntunen2009nitsche}
{\sc M.~Juntunen and R.~Stenberg}, {\em Nitsche's method for general boundary
  conditions}, Mathematics of computation, 78 (2009), pp.~1353--1374.

\bibitem{kim2015b}
{\sc T.-Y. Kim, T.~Iliescu, and E.~Fried}, {\em B-spline based finite-element
  method for the stationary quasi-geostrophic equations of the ocean}, Computer
  Methods in Applied Mechanics and Engineering, 286 (2015), pp.~168--191.

\bibitem{morin2000data}
{\sc P.~Morin, R.~H. Nochetto, and K.~G. Siebert}, {\em Data oscillation and
  convergence of adaptive fem}, SIAM Journal on Numerical Analysis, 38 (2000),
  pp.~466--488.

\bibitem{morin2002convergence}
\leavevmode\vrule height 2pt depth -1.6pt width 23pt, {\em Convergence of
  adaptive finite element methods}, SIAM review, 44 (2002), pp.~631--658.

\bibitem{morin2008basic}
{\sc P.~Morin, K.~G. Siebert, and A.~Veeser}, {\em A basic convergence result
  for conforming adaptive finite elements}, Mathematical Models and Methods in
  Applied Sciences, 18 (2008), pp.~707--737.

\bibitem{nitsche1971variation}
{\sc J.~Nitsche}, {\em {\"U}ber ein variationsprinzip zur l{\"o}sung von
  dirichlet-problemen bei verwendung von teilr{\"a}umen, die keinen
  randbedingungen unterworfen sind}, 36 (1971), pp.~9--15.

\bibitem{scott1990finite}
{\sc L.~R. Scott and S.~Zhang}, {\em Finite element interpolation of nonsmooth
  functions satisfying boundary conditions}, Mathematics of Computation, 54
  (1990), pp.~483--493.

\bibitem{siebert2010converg}
{\sc K.~G. Siebert}, {\em A convergence proof for adaptive finite elements
  without lower bound}, IMA journal of numerical analysis, 31 (2010),
  pp.~947--970.

\bibitem{speleers2016effortless}
{\em Effortless quasi-interpolation in hierarchical spaces}, {\sc Speleers, Hendrik and Manni, Carla}, {Numerische Mathematik}, {volume 132}, {number 1}, {pages 155--184}, {2016}, {Springer}

\bibitem{speleers2017hierarchical}
{\sc H. Speleers}, {\em Hierarchical spline spaces: quasi-interpolants and local approximation estimates}, {Advances in Computational Mathematics}, {volume 43}, {number 2}, {pages 235--255}, {2017}, {Springer}

\bibitem{stenberg1995some}
{\sc R.~Stenberg}, {\em On some techniques for approximating boundary
  conditions in the finite element method}, Journal of Computational and
  applied Mathematics, 63 (1995), pp.~139--148.

\bibitem{stevenson2005optimal}
{\sc R.~Stevenson}, {\em An optimal adaptive finite element method}, SIAM
  journal on numerical analysis, 42 (2005), pp.~2188--2217.

\bibitem{verfurth1994posteriori}
{\sc R.~Verf{\"u}rth}, {\em A posteriori error estimation and adaptive
  mesh-refinement techniques}, Journal of Computational and Applied
  Mathematics, 50 (1994), pp.~67--83.

\bibitem{vuong2011hierarchical}
{\sc A.-V. Vuong, C.~Giannelli, B.~J{\"u}ttler, and B.~Simeon}, {\em A
  hierarchical approach to adaptive local refinement in isogeometric analysis},
  Computer Methods in Applied Mechanics and Engineering, 200 (2011),
  pp.~3554--3567.

\end{thebibliography}
\bibliographystyle{siam}

\end{document}